\documentclass[11pt, a4paper]{amsart}
\usepackage{amssymb,amsmath,epsfig,mathrsfs, enumerate,mathtools}
\usepackage{graphicx}
\usepackage[normalem]{ulem}
\usepackage{fancyhdr}
\fancyheadoffset{0px}
\usepackage{hyperref}
\hypersetup{
 colorlinks   = true,
 urlcolor     = blue,
 linkcolor    = blue,
 citecolor   = red ,
 bookmarksopen=true
}

\newtheorem{thrm}{Theorem}[section]
\newtheorem{lemma}[thrm]{Lemma}
\newtheorem{prop}[thrm]{Proposition}

\theoremstyle{remark}
\newtheorem{rmrk}[thrm]{Remark}
\newtheorem{claim}[thrm]{Claim}
\theoremstyle{definition}
\newtheorem{dfn}[thrm]{Definition}

\usepackage[margin=3cm]{geometry}
\usepackage{enumerate}

\begin{document}

\newcommand{\Q}{\mathbb{Q}}
\newcommand{\SL}{\mathcal L^{1,p}(D)}
\newcommand{\Lp}{L^p( Dega)}
\newcommand{\CO}{C^\infty_0( \Omega)}
\newcommand{\Rn}{\mathbb R^n}
\newcommand{\Rm}{\mathbb R^m}
\newcommand{\R}{\mathbb R}
\newcommand{\Om}{\Omega}
\newcommand{\Hn}{\mathbb H^n}
\newcommand{\aB}{\alpha B}
\newcommand{\eps}{\ve}
\newcommand{\BVX}{BV_X(\Omega)}
\newcommand{\p}{\partial}
\newcommand{\IO}{\int_\Omega}
\newcommand{\bG}{\boldsymbol{G}}
\newcommand{\bg}{\mathfrak g}
\newcommand{\bz}{\mathfrak z}
\newcommand{\bv}{\mathfrak v}
\newcommand{\Bux}{\mbox{Box}}
\newcommand{\e}{\ve}
\newcommand{\X}{\mathcal X}
\newcommand{\Y}{\mathcal Y}
\newcommand{\W}{\mathcal W}
\newcommand{\la}{\lambda}
\newcommand{\vf}{\varphi}
\newcommand{\rhh}{|\nabla_H \rho|}
\newcommand{\Ba}{\mathscr B_y^{(a)}}
\newcommand{\Za}{Z_\beta}
\newcommand{\ra}{\rho_\beta}
\newcommand{\na}{\nabla_\beta}
\newcommand{\vt}{\vartheta}
\newcommand{\G}{\Gamma}
\newcommand{\Ga}{\overline{G}}
\newcommand{\BB}{\B_r}
\newcommand{\La}{\mathscr L_a}
\newcommand{\Gb}{\overline{\mathscr G}_a}
\newcommand{\Gbb}{\overline{\mathscr G}}
\newcommand{\N}{\mathbb{N}}

\numberwithin{equation}{section}

\newcommand{\RN} {\mathbb{R}^N}
\newcommand{\Sob}{S^{1,p}(\Omega)}
\newcommand{\Dxk}{\frac{\partial}{\partial x_k}}
\newcommand{\Co}{C^\infty_0(\Omega)}
\newcommand{\Je}{J_\ve}
\newcommand{\beq}{\begin{equation}}
\newcommand{\bea}[1]{\begin{array}{#1} }
\newcommand{\eeq}{ \end{equation}}
\newcommand{\ea}{ \end{array}}
\newcommand{\eh}{\ve h}
\newcommand{\Dxi}{\frac{\partial}{\partial x_{i}}}
\newcommand{\Dyi}{\frac{\partial}{\partial y_{i}}}
\newcommand{\Dt}{\frac{\partial}{\partial t}}
\newcommand{\aBa}{(\alpha+1)B}
\newcommand{\GF}{\psi^{1+\frac{1}{2\alpha}}}
\newcommand{\GS}{\psi^{\frac12}}
\newcommand{\HFF}{\frac{\psi}{\rho}}
\newcommand{\HSS}{\frac{\psi}{\rho}}
\newcommand{\HFS}{\rho\psi^{\frac12-\frac{1}{2\alpha}}}
\newcommand{\HSF}{\frac{\psi^{\frac32+\frac{1}{2\alpha}}}{\rho}}
\newcommand{\AF}{\rho}
\newcommand{\AR}{\rho{\psi}^{\frac{1}{2}+\frac{1}{2\alpha}}}
\newcommand{\PF}{\alpha\frac{\psi}{|x|}}
\newcommand{\PS}{\alpha\frac{\psi}{\rho}}
\newcommand{\ds}{\displaystyle}
\newcommand{\Zt}{{\mathcal Z}^{t}}
\newcommand{\norm}[1]{\lVert#1 \rVert}
\newcommand{\ve}{\varepsilon}
\newcommand{\Rnn}{\mathbb R^{n+1}}
\newcommand{\Rnp}{\mathbb R^{n+1}_+}
\newcommand{\Rnm}{\mathbb R^{n+1}_-}
\newcommand{\B}{\mathbb{B}}
\newcommand{\Ha}{\mathbb{H}}
\newcommand{\xa}{X}
\newcommand{\Sa}{\mathbb{S}}
\newcommand{\pa}{\p^a_y v}
\newcommand{\pax}{\p^a_y v(x,0)}
\newcommand{\Wa}{\mathscr W}
\newcommand{\ya}{|y|^a}
\newcommand{\paw}{\p^a_y w(x,0)}
\newcommand{\pau}{\p^a_y u(x,0)}
\newcommand{\pae}{\p^a_y \eta(x,0)}
\newcommand{\paeu}{\p^a_y (\eta u)(x,0)}
\newcommand{\paa}{y^a \p_y}
\newcommand{\Sp}{\mathfrak S_F(\Sa_1^+)}
\newcommand{\Ls}{E_{\ell,\sigma}}
\newcommand{\Pk}{\mathscr P_\kappa^+}
\newcommand{\Mk}{\mathscr M_\kappa(U,p_\kappa,r)}
\newcommand{\Spo}{\mathfrak S_0(\Sa_1^+)}
\def \LL {\mathscr L_a}

\title[The structure of singular set  etc.]{The structure of the singular set in the thin obstacle problem for degenerate parabolic equations}

\author{Agnid Banerjee}
\address{TIFR CAM, Bangalore-560065, India} \email[Agnid Banerjee]{agnidban@gmail.com}
\thanks{The first author was supported in part by SERB Matrix grant MTR/2018/000267}

\author{Donatella Danielli}
\address{Department of Mathematics \\ Purdue University \\West Lafayette, IN 47907, USA}
\email[Donatella Danielli]{danielli@math.purdue.edu}

\author{Nicola Garofalo}
\address{Dipartimento di Ingegneria Civile, Edile e Ambientale (DICEA) \\ Universit\`a di Padova\\ 35131 Padova, Italy}
\email[Nicola Garofalo]{rembdrandt54@gmail.com}

\thanks{The third author was supported in part by a Progetto SID (Investimento Strategico di Dipartimento) ``Non-local operators in geometry and in free boundary problems, and their connection with the applied sciences'', University of Padova, 2017.}

\author{Arshak Petrosyan}
\address{Department of Mathematics \\ Purdue University\\West Lafayette, IN 47907, USA}
\email[Arshak Petrosyan]{arshak@purdue.edu}

\thanks{The fourth author was supported in part by NSF Grant DMS-1800527.}

\begin{abstract}
We study the singular set in the thin obstacle problem for degenerate
parabolic equations with weight $|y|^a$ for $a \in (-1,1)$. Such
problem arises as the local extension of the obstacle problem for the
fractional heat operator $(\partial_t - \Delta_x)^s$ for $s \in
(0,1)$. Our main result establishes the complete structure and
regularity of the singular set of the free boundary. To achieve it, we prove
Almgren-Poon, Weiss, and Monneau type monotonicity formulas which
generalize those
for the case of the heat equation ($a=0$).
\end{abstract}
\maketitle

\tableofcontents

\section{Introduction}\label{S:Intro}

The last decade has seen a resurgence of interest in the study of lower-dimensional, or thin obstacle problems, largely motivated on the one hand by the applications, and on the other hand by the development of new mathematical tools and techniques. The primary objective of the present paper is the study of the so-called \emph{singular set} of the free boundary in the following degenerate parabolic thin obstacle problem. Given a parameter $a\in (-1,1)$, and a function $\psi$ (the thin obstacle) on $Q_1$, we consider the problem of finding a function $U$ in $\Q_1^+$ such that
\begin{equation}\label{epb}
 \begin{cases}
y^a \p_t U =\operatorname{div}_{X}(y^a \nabla_{X} U)&\text{in}\ \Q_1^+,
\\
\min\Bigl\{U(x,0,t)-\psi(x,t), -\underset{y\to 0^+}{\lim} y^a\partial_y
U(x,y,t)\Bigr\}=0, &\text{for}\ (x,t)\in Q_1.
\end{cases}
\end{equation}
For a detailed explanation of \eqref{epb} and the relevant notation we refer the reader to Section~\ref{S:prelim}. We say that \eqref{epb} is a thin obstacle problem since the function $\psi$ is supported in the codimension one manifold $\{y=0\}\times (-1,0)$ in the space-time variables $(X,t)$, with $X = (x,y) \in \Rn\times (0,\infty)$.  An important motivation for \eqref{epb} is provided by its connection to the obstacle problem for the nonlocal heat operator
\begin{equation}\label{pb}
\min\bigl\{ u-\psi,\,(\p_t -\Delta_x)^s u\bigr\}=0,
\end{equation}
with the fractional parameter $s\in (0,1)$ related to $a\in (-1,1)$ by the equation $a = 1-2s$. The passage from \eqref{pb} to \eqref{epb} rests on the extension procedure for the operator $(\p_t - \Delta_x)^s$, developed independently  by Nystr\"om and Sande in  \cite{NS} and by Stinga and Torrea in \cite{ST}. Such result represents the parabolic counterpart of the famous Caffarelli and Silvestre's extension work \cite{CS}.

When $s = 1/2$ the problem \eqref{pb} arises in the  modeling of semipermeable membranes in the process of osmosis (for this and related problems see the classical monograph \cite{DL}). In such case, by taking $a = 0$ in \eqref{epb}, we see that \eqref{pb} is equivalent to a lower-dimensional obstacle problem of Signorini type for the standard heat equation. We recall that in the paper \cite{DGPT} three of us and T.~To developed an extensive analysis for this problem. The optimal regularity of the solution was established, together with the $H^{1+\alpha,(1+\alpha)/2}$-regularity of the so-called \emph{regular} free boundary, and a structure theorem for the singular part of the free boundary. We  also refer to \cite{ACM1} for quasiconvexity results for certain generalized versions of the Signorini problem studied in \cite{DGPT}.

In the present paper, and in the work \cite{BDGP2}, we develop an  analysis similar to the one in \cite{DGPT}, but for the general case $-1<a<1$ in \eqref{epb}. In the first part of this program, which is the content of this paper,  we  provide a  systematic classification of free boundary points. The main tool is a monotonicity formula of Almgren-Poon type, which we utilize in the analysis of the blowup limits of appropriate rescalings. We also establish monotonicity formulas of Weiss- and Monneau-type, which we employ to establish a structure theorem for the singular set.

Although the work in \cite{DGPT} has served as a road map for our analysis, in the setting of the present paper one faces novel complications deriving from: a) the presence of the degenerate weight $y^a$ in \eqref{epb}; b) the lower regularity of  the solution in the time variable; and c) the fact that, because of the nature of the Almgren-Poon frequency, in the relevant $W^{2,2}$ estimates one must work with the Gaussian, instead of Lebesgue measure.

In connection with our results we recall that in their recent work \cite{ACM},  Athanasopoulos, Caffarelli, and Milakis show that, at a local level, the fractional obstacle problem \eqref{pb} is equivalent to one of type \eqref{epb} under appropriate initial and boundary conditions. Based on such correspondence,  the authors focus their attention on \eqref{epb}, establishing the optimal interior regularity of the solution, as well as the $C^{1,\alpha}$ regularity of the free boundary near certain non-singular points (which we call hyperbolic regular points, see Remark \ref{R:ell-par-hyp} for more details). In their study the authors use global assumptions on the initial
data to infer quasi-convexity properties of the solutions,
leading to their optimal regularity result. 

The present work is completely different from \cite{ACM} and it is developed in total independence from it. First of all, our main objective is the novel treatment of the singular part of the free boundary. A further difference is that our approach is purely local. By this we mean that we establish localized versions of the regularity estimates in \cite{ACM}, both for the  solution  and for the free boundary. This is of critical importance in the
further analysis of the 
problem as it allows to consider the blowups at free boundary points,
leading to their fine classification, see also our work \cite{BDGP2}, which complements and  provides a foundation for this work. 

To provide the reader with some further perspectives on the objectives of the present paper, we mention that our results are inspired by those in the time-independent case in \cite{GP}. In that paper, two of us first analyzed the structure of the singular set in the case $a = 0$ using some monotonicity formulas of Weiss and Monneau type.  More recently, their results have been extended to the whole range $a\in (-1,1)$ in \cite{GRO}. We also mention the recent interesting paper \cite{CSV}, where for the time-independent Signorini problem ($a=0$) a finer stratification of the singular set is obtained using a variant of Weiss' epiperimetric inequality, and the work \cite{FRJ} for a further refined analysis of the structure of the singular set under certain geometric assumption on the obstacle. A parabolic version of such epiperimetric inequality (again, when $a = 0$) has been very recently established in \cite{Sh}, where it has also been shown  that such an inequality, combined  with  the results  in \cite{DGPT}, provides a finer structure theorem of the singular set in the parabolic thin obstacle problem. Finally, we mention the work \cite{BG} on unique continuation for degenerate parabolic equations such as that in \eqref{epb}, where Almgren-Poon monotonicity formulas were established, and the recent work \cite{AT} for related results on the nodal sets of solution.

In closing, we say something about the organization of the present paper. In Section~\ref{S:prelim} we introduce some  basic  notations and gather some known results  which are relevant to our work.  In Section~\ref{S:classes} we introduce the class of global solutions $\Sp$ of the thin obstacle problem \eqref{epb}. In particular, we show how to effectively ``subtract'' the obstacle by maximally using its regularity, thus converting the original problem into one with zero thin obstacle, but with a non-homogeneous right hand side.  In Section~\ref{S:poon}   we establish a generalized Almgren-Poon type monotonicity formula for solutions to \eqref{epb}.  Section~\ref{S:gaussian} contains $W^{2,2}$-type estimates in the Gaussian space. Such estimates are  instrumental to the study of blowups in Section~\ref{S:blowups}, which is the most technical part of the paper. There, we prove the existence and homogeneity of blowups at free boundary points where the separation rate of the solution from the thin obstacle dominates the ``truncation'' terms in the generalized  monotonicity formula.  In Section~\ref{S:global} we establish a basic Liouville type theorem, which is  used in Section~\ref{S:classification} to classify the free boundary points according to the homogeneity of the blowup.  In Section~\ref{S:singular}  we give a characterization of the so-called singular points (i.e.,\ points where the free boundary is asymptotically negligible).  Section~\ref{S:WM} contains new Weiss- and Monneau-type monotonicity formulas which generalize those in \cite{GP}, \cite{GRO} and  \cite{DGPT}.  Finally, following the circle of ideas in \cite{DGPT} for the case $a=0$, in Section~\ref{S:structure} we briefly outline how to combine the Weiss- and Monneau-type monotonicity formulas  with the results established in the previous sections. The objective is to conclude uniqueness of blowups and obtain a structure theorem for the singular set (see Theorem~\ref{structure theorem}).  The paper ends with an appendix where we prove some of the auxiliary results  stated in Section~\ref{S:poon}, that are crucial in the proof of our Almgren-Poon type monotonicity formula.

\section{Notations and Preliminaries}\label{S:prelim}

In this section we introduce the basic notation and collect some background material which will be used throughout our work. We indicate with $x = (x_1,\ldots,x_n)$ a generic point in $\Rn$, by $(x,t)$ a point in the space-time $\Rn\times \R$, whereas the letter $y$ will denote the ``extension variable" on the half-line $(0,\infty)$. The generic point in $\Rnn_+ = \Rn\times (0,\infty)$ will be denoted by $X = (x,y)$. At times, we will tacitly use the same notation to indicate the generic point in $\Rnn$, i.e.,\ without the restriction that $y$ be $>0$. For instance, given $r>0$ we respectively denote by $B_r$ and $\B_r$ the Euclidean balls centered at the origin with radius $r$ in the variables $x\in \Rn$ and $X = (x,y)\in \Rnn$. We also let $\B_r^\pm = \{X = (x,y)\in \B_r\mid \pm y>0\}$.
 We denote by
\[
Q_r = B_r \times (-r^2,0],\quad\mathbb Q_r = \B_r \times (-r^2,0],\quad r>0,
\]
respectively the parabolic cylinders in the thin space $(x,t) \in \Rn\times \R$ and thick space $(X,t)\in \Rnn\times \R$.  We will indicate by
\[
\mathbb Q^\pm_r = \B^\pm_r \times (-r^2,0]
\]
the parabolic half-cylinder in the thick space.

Given an open set $E \subset\Rnp \times \R$ and $m \in \mathbb{N}$,  by $W^{2m,m}_q(E, y^a dXdt)$ we will denote the parabolic Sobolev space of functions $u$ in $L^{q}(E, y^a dXdt)$ whose distributional derivatives $\partial_t^{\alpha} \partial_X^{\beta}u$ belong to $L^q(E, y^a dXdt)$ for $2|\alpha|+ |\beta| \leq 2m$. Such a space is endowed with the natural norm. Further, for given $k\in\N\cup\{0\}$ and $0<\alpha\leq1$ by $H^{k+\alpha, (k+\alpha)/2}$ we will indicate the classical parabolic H\"older spaces, see e.g.\ \cite{DGPT} for detailed definition.

Given a number $a\in (-1,1)$, we consider in $\Rnn\times \R$ the degenerate parabolic operator defined by
\begin{equation}\label{extop}
\mathscr L_a U \overset{\rm def}{=}\p_t(|y|^a U) - \operatorname{div}_X(|y|^a \nabla_X U).
\end{equation}
This is the so-called \emph{extension operator} for the fractional powers  $(\p_t - \Delta_x)^s$, $0<s<1$, of the heat operator.
It was recently introduced independently by Nystr\"om-Sande in \cite{NS}, and Stinga-Torrea in \cite{ST}. These authors proved that, if for a given $u\in \mathscr S(\Rnn)$, the function $U$ solves the problem
\[
\begin{cases}
\LL U = 0&\text{in}\ \Rnp\times (0,\infty),
\\
U(x,0,t) = u(x,t),& (x,t)\in\Rn\times(0,\infty),
\end{cases}
\]
(such problem can be solved by means of an explicit Poisson kernel) then, with $s\in (0,1)$ determined by the equation $a = 1-2s$, one has (both in $L^\infty$ and $L^2$)
\[
- \frac{2^{-a}
  \Gamma\left(\frac{1-a}2\right)}{\Gamma\left(\frac{1+a}2\right)}\,\partial_y^a
U(x,0,t) = (\p_t -\Delta_x)^s u(x,t),
\]
where $\partial_y^a U$ denotes the weighted normal derivative
$$
\partial_y^a U(x,0,t)\overset{\rm def}{=}\lim_{y\to 0^+}y^a \p_y U(x,y,t).
$$
The proof is based on the representation
\begin{equation}\label{extop2}
\LL = y^a(\p_t - \Delta_x - \Ba),\quad\text{for }y>0,
\end{equation}
where $\Ba = \p_y^2 + (a/y)\p_y$ is the generator of the Bessel semigroup on $(\R^+,y^a dy)$.
Moreover, it was shown in \cite{ACM} that, at a local level, problem \eqref{pb} is equivalent to
the following thin obstacle problem for the local degenerate parabolic
equation
\begin{equation}\label{epb2}
 \begin{cases}
\LL U =0&\text{in}\ \Q_1^+,
\\
\min\{U(x,0,t)-\psi(x,t), -\partial_y^a U(x,0,t)\}=0 &\text{on}\ Q_1,
\end{cases}
\end{equation}
which is the same as \eqref{epb}.
Although such denomination is commonly used for the case $a = 0$ ($s = 1/2$), throughout the paper we will routinely refer to \eqref{epb2} as the parabolic \emph{Signorini problem}.
We will also assume that the solution of  has the following minimal regularity:
\begin{itemize}
\item $\nabla_x U, y^a U_y \in H^{\alpha, \alpha/2}(\Q_1^+)$ for some $\alpha>0$;
\item $U_t \in L^{\infty}(\Q_1^{+})$;
\item $y^a|\nabla U_{x_i}|^2, y^{-a} (y^a U_y)_y^2 \in L^{1}(\Q_1^{+})$.
\end{itemize}
This regularity follows for instance from global semiconvexity assumptions in \cite{ACM}, with norms depending on the initial data. But, for solutions of \eqref{epb2}, it can also be obtained directly in the form of interior estimates independent of initial data, see the forthcoming paper \cite{BDGP2}.

We next consider, for any $a> - 1$, the Cauchy problem
 with Neumann boundary condition
\begin{equation}\label{CP}
\begin{cases}
\p_t u  - \Ba u = 0 & \text{in }(0,\infty)\times(0,\infty),
\\
u(y,0) = \vf(y), & y\in (0,\infty),
\\
\partial_y^a u(0,t) = 0, & t\in (0,\infty).
\end{cases}
\end{equation}
This corresponds to one-dimensional Brownian motion reflected at $y = 0$.
Consider the following classes of functions
\[
\mathscr C_{(a)}(0,\infty) = \left\{\vf\in C(0,\infty) \,\middle|\, \int_0^R |\vf(y)| y^{a} dy < \infty,
\ \int_R^\infty |\vf(y)| y^{\frac a2} dy < \infty, \forall R>0\right\},
\]
and
\[
\mathscr C^1_{(a)}(0,\infty) = \left\{\vf\in C^1(0,\infty)\mid \vf, y^{-1} \vf' \in \mathscr C_{(a)}(0,\infty)\right\}.
\]
As it was observed in (22.8) of \cite{G} membership in $\mathscr C^1_{(a)}(0,\infty)$ imposes, in particular, the weak \emph{Neumann condition}
\begin{equation}\label{weakn}
\underset{y\to 0^+}{\liminf}\ y^{a} |\vf'(y)| = 0.
\end{equation}

For an analytic proof of the next result we refer the reader to Proposition~22.3 in \cite{G}.

\begin{prop}\label{P:CP}
Given $\vf\in\mathscr C^1_{(a)}(0,\infty)$, the Cauchy problem \eqref{CP} admits the following solution
\begin{equation}\label{repfor}
u(y,t) = P^{(a)}_t \vf(y) \overset{\rm def}{=} \int_0^\infty \vf(\eta) p^{(a)}(y,\eta,t) \eta^a d\eta,
\end{equation}
where for $y,\eta,t>0$ we have denoted by
\begin{align}\label{fs}
p^{(a)}(y,\eta,t) & =(2t)^{-\frac{a+1}{2}}\left(\frac{y\eta}{2t}\right)^{\frac{1-a}{2}}I_{\frac{a-1}{2}}\left(\frac{y\eta}{2t}\right)e^{-\frac{y^2+\eta^2}{4t}}.
\end{align}
For $t\le 0$ we set $p^{(a)}(y,\eta,t) \equiv 0$.
\end{prop}

In \eqref{fs} we have denoted by $I_\nu(z)$ the modified Bessel function of the first kind and order $\nu\in \mathbb C$ defined, in the complex plane cut along the negative real axis, by the series
\begin{equation}\label{I}
I_\nu(z) = \sum_{k=0}^\infty \frac{(z/2)^{\nu+2k}}{\G(k+1) \G(k+\nu+1)},\qquad |z|<\infty,\ |\arg z| < \pi.
\end{equation}
When restricted to the positive real axis $\Re z>0$, as in \eqref{fs} above, the function $I_\nu$ takes strictly positive values for every $\nu>-1$. As a consequence of this observation, $I_{\frac{a-1}{2}}\left(\frac{y\eta}{2t}\right)>0$ for every $a>-1$, and every $y, \eta, t>0$. We note the following elementary properties of the Bessel heat kernel $p^{(a)}$:
\begin{itemize}
\item[(i)] $p^{(a)}(y,\eta,t)>0$ for every $y, \eta>0$ and $t>0$;
\item[(ii)] $p^{(a)}(y,\eta,t) = p^{(a)}(\eta,y,t)$;
\item[(iii)] $p^{(a)}(\la y,\la \eta,\la^2 t) = \la^{-(a+1)} p^{(a)}(y,\eta,t)$.
\end{itemize}
By Remark 22.4 in \cite{G}, for every $y>0, t>0$ one has
\begin{equation}\label{bfs0}
p^{(a)}(y,t)  \overset{\rm def}{=} p^{(a)}(y,0,t)  = \frac{1}{2^a \G(\frac{a+1}{2})} t^{-\frac{a+1}{2}} e^{-\frac{y^2}{4t}}.
\end{equation}

The next two results show that \eqref{repfor} defines a stochastically complete semigroup $\{P^{(a)}_t\}_{t>0}$. For their proofs we refer to \cite[Propositions 2.3, 2.4]{Gbessel}.

\begin{prop}\label{P:sc}
Let $a>-1$. For every $(y,t)\in (0,\infty)\times (0,\infty)$ one has
\[
\int_0^\infty p^{(a)}(y,\eta,t) \eta^a d\eta = 1.
\]
\end{prop}

\begin{prop}\label{P:semigroup}
Let $a>-1$. For every $y, \eta>0$ and every $0<s, t<\infty$ one has
\[
p^{(a)}(y,\eta,s+t) = \int_0^\infty p^{(a)}(y,\zeta,t) p^{(a)}(\zeta,\eta,s) \zeta^a d\zeta.
\]
\end{prop}

We further note that in view of representation \eqref{extop2}, the fundamental solution for $\LL$ in $\Rnp\times(0,\infty)$, with Neumann condition on the thin manifold $(\Rn\times\{0\})\times (0,\infty)$, and singularity at $(Y,0) = (\xi,\eta,0)$, is given by
\begin{equation}\label{sGa}
\mathscr G_a(X,Y,t) = p(x,\xi,t) p^{(a)}(y,\eta,t),
\end{equation}
where $p(x,\xi,t) = (4\pi t)^{-n/2} \exp(-\frac{|x-\xi|^2}{4t})$ is the standard heat kernel in $\Rn\times (0,\infty)$ and $p^{(a)}(y,\eta,t)$ is given by \eqref{fs} above. This means that, given a function $\vf\in C^\infty_0(\Rnp)$, the Cauchy problem with Neumann condition
\begin{equation}\label{cpext}
\begin{cases}
\LL U = 0&\text{in}\ \Rnp\times (0,\infty)
\\
U(X,0) = \vf(X),&X\in \Rnp,
\\
\partial_y^a U(x,0,t) = 0, & t\in (0,\infty),

\end{cases}
\end{equation}
is represented by the formula
\begin{equation}\label{solform}
U(X,t) = \int_{\Rnp} \vf(Y) \mathscr G_a(X,Y,t) \eta^a dY.
\end{equation}
Using Proposition~\ref{P:sc}, and the well-known fact that $\int_{\Rn} p(x,\xi,t) d\xi = 1$ for every $x\in \Rn$ and $t>0$, it is  trivial to verify that for every $X\in \Rnp$ and $t>0$ one has
\begin{equation}\label{stochcompl}
\int_{\Rnp} \mathscr G_a(X,Y,t) \eta^a dY = 1.
\end{equation}
We also note that (ii) and (iii) above give for every $X, Y \in \Rnp$, and $t>0$,
\begin{itemize}
\item[(ii)$'$] $\mathscr G_a(X,Y,t) = \mathscr G_a(Y,X,t)$,
\item[(iii)$'$] $\mathscr G_a(\la X,\la Y,\la^2 t) = \la^{-(n+a+1)} \mathscr G_a(X,Y, t)$.
\end{itemize}
Henceforth, we take $Y = 0$ in \eqref{sGa}, and with a slight abuse of the notation, we write
\[
\mathscr G_a(X,t) = \mathscr G_a(X,0,t).
\]
By \eqref{sGa} and \eqref{bfs0} above, we obtain
\begin{equation}\label{sGa2}
\mathscr G_a(X,t)  = \frac{(4\pi)^{-\frac n2}}{2^a \G(\frac{a+1}{2})} t^{-\frac{n+a+1}{2}} e^{-\frac{|X|^2}{4t}}.
\end{equation}
From \eqref{stochcompl} and (ii)$'$ we have for every $t>0$
\begin{equation}\label{stochcompl2}
\int_{\Rnp} \mathscr G_a(X,t) y^a dX = 1.
\end{equation}
We denote by
\begin{equation}\label{Gbarra}
\Gb(X,t) = \mathscr G_a(X,|t|), \quad t<0,
\end{equation}
the Neumann fundamental solution of the backward operator $\La^\star = y^a \frac{\p}{\p t} + \operatorname{div}_{X}(y^a \nabla_{X})$. This means that $\Gb$ satisfies the equation in $\Rnp\times (-\infty,0)$,
\begin{equation}\label{gb}
\La^\star \Gb = y^a \p_t \Gb + \operatorname{div}_{X}(y^a \nabla_{X} \Gb) = 0,
\end{equation}
plus the Neumann condition
\begin{equation}\label{nc}
  \partial_y^a\Gb(x,0,t) = 0.
\end{equation}
From \eqref{sGa2}, for $X\in \Rnp$ and $t<0$ we have the reproducing property
\begin{equation}\label{rp}
\nabla \Gb= \frac{X}{2t} \Gb.
\end{equation}
We now consider the parabolic dilations in $\Rnn\times \R$ defined by
\begin{equation}\label{pardil}
\delta_\la(X,t) = (\la X,\la^2 t).
\end{equation}
A function $f:\Rnn\times \R\to \R$ is said to be homogeneous of degree $\kappa\in \R$ with respect to \eqref{pardil} if $f\circ \delta_\la = \la^\kappa f$. The infinitesimal generator of the group $\{\delta_\la\}_{\la>0}$ is
\begin{equation}\label{z1}
Zf= \langle X,\nabla f\rangle   + 2tf_t.
\end{equation}
A $C^1$ function is $\kappa$-homogeneous with respect to \eqref{pardil} if and only if one has $Zf = \kappa f$.
For instance, since from (iii)$'$ above we see that
\begin{equation}\label{Ghom}
\Gb \circ \delta_\la = \la^{-(n+a+1)} \Gb,
\end{equation}
and therefore
\begin{equation}\label{ZGhom}
Z \Gb = - (n+a+1) \Gb.
\end{equation}
For later use we notice that for every $(X,t)$ such that $t\not= 0$, \eqref{z1} can be rewritten
\begin{equation}\label{z11}
\frac{Zf}{2t} = f_t + \langle \nabla f,\frac{X}{2t}\rangle  .
\end{equation}
Further, we indicate with
\[
|(X,t)| = \sqrt{|X|^2 + |t|},
\]
the standard parabolic pseudo-distance from the origin in the variables $(X,t)\in \Rnn\times \R$. Notice that such function is positively homogeneous of degree one with respect to the dilations \eqref{pardil}.

In closing, for every $r>0$ we introduce the sets
\begin{equation}\label{strips}
\begin{aligned}
\Sa_r&= \Rnn\times(-r^2,0],
 \\
\Sa_{r}^{+} &=\Rnp\times(-r^2,0],
\\
S_r &=\Rn\times(-r^2,0].
 \\
 \end{aligned}
 \end{equation}
We emphasize that the $+$ sign in the notation $\Sa_{r}^{+}$ refers to the variable $y>0$ and not to the time variable $t$, which is instead negative for points in such set. The following simple lemma will be used in the subsequent sections.

\begin{lemma}\label{L:simplefact}
For every $r>0$ we have
\[
\frac 1{r^2} \int_{\Sa_r^+} \Gb y^a dX dt = 1.
\]
\end{lemma}

\begin{proof}
By \eqref{Ghom} we have
\begin{align*}
\int_{\Sa_r^+} \Gb y^a dX dt &= r^{-(n+a+1)} \int_{\Sa_r^+} \Gb(X/r,t/r^2) y^a dX dt\\
& = r^2 \int_{-1}^0 \int_{\Rnp} \Gb(Y,\tau) \eta^a dY d\tau = r^2,
\end{align*}
where in the second equality we have made a change of variables $Y = X/r$, $\tau = t/r^2$, for which $y^a dX dt = r^{n+a+3} \eta^a dY d\tau$, and in the last equality we have used \eqref{stochcompl2}.
\end{proof}

\section{Classes of solutions}\label{S:classes}

In this section we make some critical reductions on the problem
\eqref{epb2}.
As a first step, we reduce the problem \eqref{epb2} to one with zero obstacle at the expense of introducing a nonzero right-hand side in the governing equation.  The most straightforward way to do so is by considering the difference
\begin{equation}\label{zo}
W(X,t)  = U(X,t) - \psi(x,t).
\end{equation}
Later on, in order to take advantage of a possible higher regularity of $\psi$, we will make a more refined construction.
Since $U$ solves \eqref{epb2}, we have in $\Q_1^+$
\[
\La W = \La U - \La \psi = y^a \tilde F,
\]
where we have let
\[
\tilde F(X,t) = \tilde F(x,t)  \overset{\rm def}{=} - (\p_t - \Delta_x)\psi(x,t).
\]
For later purposes it is important that we note here that the function $\tilde F$, being independent of the variable $y$, is automatically even in such variable.
If we now assume that $\psi\in C^{1,1}_x$ and $\psi\in C^{0,1}_t$, then we clearly have $\tilde F \in L^\infty(\Rnn\times \R)$.
We thus see that the function $W$ satisfies
\begin{equation}\label{epb222}
 \begin{cases}
\La W = y^a \tilde F&\text{in}\ \Q_1^+,
\\
\min\{W(x,0,t), -\partial_y^a W(x,0,t)\}=0 &\text{on }Q_1.
\end{cases}
\end{equation}
We next want to extend \eqref{epb222} to a problem in a strip $\Sa_1^+$.
Pick a cut-off function $\zeta \in C^\infty_0( B_1 \times (-1,1) )$  of the type $\zeta(X)= \zeta_1 (x) \zeta_2(y)$ with $0\le \zeta_1, \zeta_2 \le 1$, and such that $\zeta_1 \equiv 1$ in $B_{3/4} $, $\zeta_2 \equiv 1$ in $(-3/4, 3/4)$. Moreover we can choose $\zeta_1, \zeta_2$ such that $\zeta_1$ is a function of $|x|$ and $\zeta_2$ is symmetric in $y$. We now let
\begin{equation}\label{cut}
V(X,t) = \zeta(X) W(X,t)=\zeta(X)(U(X,t)-\psi(x,t)).
\end{equation}
Clearly, $V$ is supported in $\overline{\Q_1^+}$. Since $\zeta$ is smooth and symmetric in $y$, the function $V$ will satisfy on the thin set $S_1$ the same Neumann condition as $W$. Furthermore, we have $\zeta_y = O(|y|)$ near the thin set $\{y=0\}$, which implies that $y^{-a} \operatorname{div}(y^a\nabla\zeta) = O(1)$ and $\zeta_y V_y  = O(1)$ up to the thin set. Therefore, if we let
\[
F \overset{\rm def}{=} \zeta \tilde F  - V y^{-a} \operatorname{div}(y^a\nabla\zeta) - 2 \langle \nabla V,\nabla \zeta\rangle  ,
\]
then $F\in L^\infty(\Sa_1^+)$ and $V$ solves the problem
\begin{equation}\label{gn}
 \begin{cases}
\La V = y^a F&\text{in}\ \Sa_1^+,
\\
\min\{V(x,0,t), -\partial_y^a V(x,0,t)\}=0 &\text{on }S_1.
\end{cases}
\end{equation}
Recalling now the minimal regularity assumptions imposed on the solutions of \eqref{epb2}, we are ready to introduce a central class of solutions in this paper.

\begin{dfn}[Solutions in strips]\label{cls}
Given a function for $F \in L^{\infty}(\Sa_1^{+})$, we say that $U \in \Sp$  if:
\begin{itemize}
\item[1)] $U$ has bounded support;
\item[2)] $\nabla_x U, y^a U_y \in H^{\alpha, \alpha/2}(\Sa_1^+)$ for some $\alpha>0$;
\item[3)] $U_t \in L^{\infty}(\Sa_1^{+})$;
\item[4)] $y^a|\nabla U_{x_i}|^2, y^{-a} (y^a U_y)_y^2 \in L^{1}(\Sa_1^{+})$;
\item[5)] $U$ solves \eqref{gn};
\item[6)] $(0,0) \in \Gamma_{*}(U)\overset{\rm def}{=}\partial \{(x,t) \in S_1 \mid U(x,0,t)=0,\  \p^a_y U(x, 0,t)=0\}$.
\end{itemize}
\end{dfn}

Suppose now the obstacle  $\psi$ in  \eqref{epb2} is of class  $ H^{\ell, \ell/2} (Q_1)$ with $\ell= k+\gamma \geq 2$, $k \in \mathbb N$,  $0 < \gamma \leq 1$. We then make the following more refined construction that takes advantage of the higher regularity of $\psi$. Let $q_{k}(x,t)$ be the parabolic Taylor polynomial of $\psi$ at the origin  of parabolic degree $k$. Then, we have
\[
|\psi(x, t) - q_k (x, t)| \leq C |(x,t)|^{\ell},
\]
and more generally
\begin{equation}\label{dc}
|\partial_{x}^{\alpha} \partial_{t}^{j} ( \psi - q_k)| \leq M |(x,t)|^{\ell-|\alpha|- 2j},
\end{equation}
for any multi-index $\alpha$ and $j\geq 0$ with $|\alpha|+2j\leq k$.
We then extend the polynomial $q_k$ into $\Rnn\times\R$ as an $a$-caloric polynomial, even in $y$, with the help of the following lemma.

\begin{lemma}[$a$-Caloric extension of polynomials]\label{calext}
For a given polynomial $q(x,t)$ in $\Rn \times \R$, there exists a unique polynomial $\tilde q(x,y,t)$ in $\Rnn \times \R$, which satisfies
\begin{equation}
\begin{cases}
\tilde q(x, 0, t)= q(x, t), & (x,t)\in\Rn\times\R\\
\tilde q(x,-y,t)=\tilde q(x,y,t), &(x,y,t)\in\Rnn\times\R\\
\La \tilde q=0, &\text{in }\Rnn\times\R.
\end{cases}
\end{equation}
Moreover, if $q(x,t)$ is parabolically homogeneous of degree $\kappa$, then $\tilde q$ has the same homogeneity.
\end{lemma}

\begin{proof}
 The proof is similar to that of Lemma~4.3 in \cite{DGPT} for the case $a=0$ and Lemma~5.2 in \cite{GRO} for the stationary case.  For a given polynomial $q(x,t)$, let
 \[
 \tilde q(x,y,t)=\sum_{k\geq 0} (-1)^k c_{k} (\Delta_x - \partial_t)^k  q(x,t) y^{2k},\quad\text{with }
 c_{k}= \prod_{i=1}^k \frac1{2i(2i-2s)},\ c_0=1.
 \]
Note that the sum above runs over a finite range of $k$, with $2k$ not exceeding the parabolic degree of $q(x,t)$. It is clear that
\[
\tilde q(x,0,t)=q(x,t)
\]
and that $\tilde q$ is even in $y$. Further, using that
\[
\left(\partial_{y}^2+\frac{a}{y}\partial_y\right)(c_k y^{2k})=c_{k-1}y^{2(k-1)},
\]
(with the agreement that $c_{-1}=0$)
it is straightforward to check that
\[
\LL \tilde q(X,t) = y^a(\p_t - \Delta_x - \Ba) \tilde q(X,t)=0.
\]
Hence, $\tilde q$ is the required $a$-caloric extension of $q$, even in $y$.
We next show the uniqueness of such extension. By linearity of $\La$, it suffices to show that the only extension of $q=0$ is $\tilde q=0$.  Note that for any such extension, both $\tilde q$ and $\p_y^a \tilde q$ vanish on $\{y=0\}$. Now, from the strong unique continuation property (which follows  by applying  the arguments in Lemma~7.7 in \cite{BG}), we conclude that $\tilde q \equiv 0$.
\end{proof}

Let now $q_k$ be the parabolic Taylor polynomial of $\psi$ of parabolic degree $k$, and $\tilde q_k$ be the corresponding $a$-caloric extension as in Lemma~\ref{calext}. Consider
\[
U_k = U - \tilde q_k(X,t),\quad\psi_k= \psi - q_k(x, t),
\]
where $U$ is as in \eqref{epb2}. It is easy to see that $U_k$ solves the thin obstacle problem with the thin obstacle $\psi_k$. With $\zeta$ a cut-off function as in \eqref{cut}, we now consider
\begin{equation}\label{global}
V_k = \zeta(X)( U_k - \psi_k).
\end{equation}
Then, $V_k$  is a global solution to the  Signorini problem \eqref{gn}, corresponding to a right-hand side $F_k$ given by
\[
F_k = \zeta (\Delta_x \psi_k - \partial_t \psi_k)   - V_k |y|^{-a} \operatorname{div}(y^a\nabla\zeta) - 2 \langle \nabla V_k,\nabla \zeta\rangle.
\]
Since $\zeta \equiv 1$ in a neighborhood of $0$, from \eqref{dc} we obtain that $F_k$ satisfies when $\ell\geq 2$,
\begin{alignat}{2}\label{fbound}
|F_k(X,t)| &\leq  M |X,t)|^{\ell -2} &\quad&\text{for }(X,t)\in\Sa_1^+.\\
\intertext{If $\ell\geq 3$ we will also have}
\label{nablafbound}
|\nabla_X F_k(X, t)| &\leq M |X, t|^{\ell -3} &&\text{for }(X,t)\in\Q_{1/2}^+.\\
\intertext{For $\ell\geq 4$ we will gain}
\label{ptfbound}
 |\p_t F_k(X, t)| &\leq M |X, t|^{\ell -4} && \text{for }(X,t)\in\Q_{1/2}^+.
\end{alignat}
Moreover, since $V_k(x, 0, t)= U(x, 0, t)- \psi(x,t)$ and $\p_y^a V_k(x, 0, t)= \p_y^a U(x, 0, t)$ in $Q_{1/2}$, it follows  that $\Gamma_{*}(V_k)= \Gamma_{*} (U) $ in  $Q_{1/2}$.

With the help of the monotonicity formulas that we prove in the next section, the growth estimates \eqref{fbound}--\eqref{ptfbound}
will allow a finer classification of free boundary points.

\section{Almgren-Poon type monotonicity formula}\label{S:poon}

In this section we establish a monotonicity formula which plays an essential role in our classification of free boundary points.
We consider a function $U\in \Sp$. In view of \eqref{gn}, this means in particular that $U$ solves the equation
\begin{equation}\label{pareq}
y^a \p_t U - \operatorname{div}_X(y^a \nabla_X U) = y^a F\quad\text{in }\Sa_1^+.
\end{equation}
We assume henceforth that the function $F$ satisfies for some $\ell\geq 2$ and a constant $C_\ell$,
\begin{equation}\label{aF}
|F(X,t)| \le C_\ell |(X,t)|^{\ell -2}\quad \text{for every $(X,t)\in \Sa_1^+$}.
\end{equation}
Recall that, when the obstacle is of class $H^{\ell,\ell/2}$, such assumption can be ensured by the reduction argument in Section~\ref{S:classes}, see \eqref{fbound}.
We also note that, because of the technical nature of the results in this section, some of the proofs are deferred to the appendix in Section~\ref{S:appA}.

For $t<0$ we introduce the quantities
\begin{align}\label{hsmall}
h(U,t)&=  \int_{\Rnp} U(X,t)^2\ \Gb(X,t)\ y^{a} dX,\\
\label{dsmall}
d(U, t)&= - t \int_{\Rnp} |\nabla U(X,t)|^2\ \Gb(X,t) y^a dX,\\
\intertext{and}
\label{ismall}
i(U,t) &= \frac 12  \int_{\Rnp}   U(X,t) ZU(X,t) \ \Gb(X,t) y^{a} dX,
\end{align}
where $Z$ is the vector field in \eqref{z1} above.
Henceforth, we will routinely drop the indication of the variables $(X,t)$ and of the $(n+1)$-dimensional Lebesgue measure $dX$ in all integrals involved.
We will need the following result connecting $d(U,t)$ and $i(U,t)$. For the proof, see Section~\ref{S:appA}.

\begin{lemma}\label{L:alti}
For $t\in (-1,0)$ we have
\begin{equation}\label{alti}
 i(U,t)  = d(U,t) - \int_{\Rnp} |t| U F \Gb y^a + \int_{\Rn\times \{0\}} |t| U \p^a_y U \Gb.
\end{equation}
\end{lemma}

Next, we introduce the following Steklov-type averaged versions of the quantities $h(U,t)$, $d(U,t)$, and  $i(U,t)$:
\begin{align}\label{HU}
H(U,r) &= \frac{1}{r^2} \int_{-r^2}^0 h(U,t) dt = \frac{1}{r^2} \int_{\Sa_r^+} U^2\ \Gb y^{a} dX dt,\\
\label{IU}
D(U,r) &= \frac{1}{r^2} \int_{-r^2}^0 d(U,t)  dt = \frac{1}{r^2} \int_{\Sa_r^+} |t| |\nabla U|^2 \ \Gb y^{a} dXdt,\\
\intertext{and}
\label{IIU}
I(U,r) &= \frac{1}{r^2} \int_{-r^2}^0 i(U,t)  dt = \frac{1}{2r^2} \int_{\Sa_r^+} U ZU\ \Gb y^a.
\end{align}
We now define two initial \emph{frequencies} of $U$ that will  each prove useful in the computations.
\begin{equation}\label{fU}
N(U,r) = 2\frac{I(U,r)}{H(U,r)},\quad\tilde N(U,r) = 2\frac{D(U,r)}{H(U,r)}.
\end{equation}

\begin{rmrk}\label{R:freq}
We remark that if $U\in \Spo$ is homogeneous of degree $\kappa$ with respect to the dilations \eqref{pardil}, then we have
\[
N(U,r) = \tilde N(U,r) \equiv \kappa.
\]
In fact, since $F\equiv 0$ we have $I(U,r) = D(U,r)$ from Lemma~\ref{L:IIU}. But then, keeping in mind that   $ZU = \kappa U$, we find from \eqref{IIU}
\[
I(U,r) = \frac{\kappa}2 H(U,r).
\]
This proves the claim.
\end{rmrk}

Using Lemma~\ref{L:alti} we immediately obtain the following alternative expression for $I(U,r)$.

\begin{lemma}\label{L:IIU}
One has for every $r\in (0,1)$
\[
I(U,r) = D(U,r) - \frac{1}{r^2} \int_{\Sa_r^+} |t| U F\ \Gb y^{a} dXdt.
\]
\end{lemma}

We now list two key results: the first-variation formulas for $H(U,r)$ and $I(U,r)$. Their proofs are given in Section~\ref{S:appA}.

\begin{lemma}[First variation of the height]\label{L:HU'}
For a.e.\ $r\in (0,1)$ we have
\[
H'(U,r) =  \frac 4r I(U,r).
\]
\end{lemma}

We observe that combining Lemma~\ref{L:HU'} with the former identity in \eqref{fU}, for every $r\in (0,1)$ such that $H(U,r)>0$ we can write
\begin{equation}\label{fU2}
N(U,r)=\frac{r H'(U,r)}{2H(U,r)} .
\end{equation}

We will need the following result.

\begin{lemma}\label{L:nondeg}
For every $r\in (0,1)$ such that $H(U,r)>0$, one has
\[
1 + N(U,r) \ge 0.
\]
\end{lemma}

\begin{proof}
From \eqref{HU} we have
\[
\frac{d}{dr}\left(r^2 H(U,r)\right) = \frac{d}{dr} \int_{-r^2}^0 \int_{\Rnp} U^2\ \Gb y^{a} dX dt = 2 r \int_{\Rnp\times \{-r^2\}} U^2\ \Gb y^{a} dX\geq 0.
\]
If $H(U,r)>0$, this gives
\begin{align*}
0 \le 2r H(U,r) + r^2 H'(U,r) &= 2 r H(U,r) \left(1 + \frac{r H'(U,r)}{2 H(U,r)}\right)\\&=  2 r H(U,r) \left(1 +  N(U,r)\right),
\end{align*}
which implies the statement of the lemma.
\end{proof}

For later use in the proof of Theorem~\ref{T:poon} we also record the following notable consequence of the above computation
\begin{equation}\label{record}
\int_{\Rnp\times \{-r^2\}} U^2\ \Gb y^{a} dX = H(U,r) \left(1 +  N(U,r)\right).\qedhere
\end{equation}

\begin{lemma}[First variation of the energy]\label{L:DU'}
For a.e.\ $r\in (0,1)$ we have
\[
D'(U,r) =\frac{1}{r^3} \int_{\Sa_r^+} (ZU)^2\ \Gb y^a + \frac{2}{r^3} \int_{\Sa_r^+} |t| (ZU) F\ \Gb y^a.
\]
\end{lemma}
Combining Lemma~\ref{L:IIU}
 with Lemma~\ref{L:DU'} we immediately obtain the following result, see also Section~\ref{S:appA}.

\begin{lemma}[First variation of the total energy]\label{L:IU'}
For a.e.\ $r\in (0,1)$ we have
\begin{align*}
I'(U,r) & =\frac{1}{r^3} \int_{\Sa_r^+} (ZU)^2\ \Gb y^a + \frac{2}{r^3} \int_{\Sa_r^+} |t| (ZU) F\ \Gb y^a
\\
& \qquad + \frac{2}{r^3} \int_{\Sa_r^+} |t| U F\ \Gb y^a + 2 r \int_{\Rnp\times\{-r^2\}} U F\ \Gb y^a.
\end{align*}
\end{lemma}

With the statement of Lemmas \ref{L:HU'} and \ref{L:IU'} in place we now establish a basic monotonicity formula that plays a central role in our classification of free boundary points.

\begin{thrm}[Monotonicity formula of Almgren-Poon type] \label{T:poon}
Let $U \in \Sp$ with $F$ satisfying \eqref{aF}. Then, for every $\sigma\in (0,1)$ there exist a constant $C>0$, depending on $n, a, C_\ell$ and $\sigma$,  such that the function
\begin{equation}\label{parfreq}
r\mapsto \Phi_{\ell,\sigma}(U,r) \overset{\rm def}{=} \frac{1}{2} r e^{C r^{1-\sigma}} \frac{d}{dr}\log \max\left\{H(U,r),\ r^{2\ell - 2+2\sigma}\right\} + 2(e^{Cr^{1-\sigma}}-1),
\end{equation}
is monotone nondecreasing on $(0,1)$.
In particular, the following limit exists
\[
\Phi_{\ell,\sigma}(U,0^+) \overset{\rm def}{=} \lim_{r\to 0^+}\Phi_{\ell,\sigma}(U,r).
\]
\end{thrm}

\begin{proof}
We begin by introducing the set
\[
\Ls = \{r\in (0,1)\mid H(U,r) > r^{2\ell - 2 + 2 \sigma}\}.
\]
As it is well known by now, in order to prove the theorem it suffices to verify the monotonicity of the function $r\to \Phi_\sigma(U,r)$ in the set $\Ls$. In such set we have
\begin{align*}
\Phi_{\ell,\sigma}(U,r) & = \frac{1}{2} r e^{C r^{1-\sigma}} \frac{d}{dr}\log H(U,r)+2(e^{Cr^{1-\sigma}}-1) = \frac{1}{2}r e^{C r^{1-\sigma}} \frac{H'(U,r)}{H(U,r)} + 2(e^{Cr^{1-\sigma}}-1)
\\
& =  e^{C r^{1-\sigma}}(N(U,r) + 2) - 2,
\end{align*}
where in the last equality we have used  \eqref{fU2}.
We now make the crucial observation that, thanks to Lemma~\ref{L:nondeg}, we can say that $r\to N(U,r) + 2>0$ in $\Ls$. Therefore, to complete the proof it suffices to show that we have in $\Ls$
\[
\frac{d}{dr} \log \Phi_{\ell,\sigma}(U,r) \ge 0.
\]
Finally, this is equivalent to proving that for every $r\in \Ls$ we have
\begin{equation}\label{newidea}
\frac{(N(U,r) + 2)'}{N(U,r) + 2} \ge - \bar C r^{-\sigma},
\end{equation}
for some constant $\bar C>0$ depending on $n, a, C_\ell, \sigma$. Then, the thesis of the theorem will follow with $C = \bar C/(1-\sigma)>0$.
We thus turn to proving \eqref{newidea}.

Using the first equation in \eqref{fU} we find
\begin{align*}
\frac{r^5}{2} H(U,r)^2 N'(U,r) & = r^5 H(U,r)^2 \frac{I'(U,r)H(U,r) - I(U,r) H'(U,r)}{H(U,r)^2}
\\
& = r^5\big(I'(U,r)H(U,r) - I(U,r) H'(U,r)\big)
\\
& = r^5 H(U,r) \left(I'(U,r) - \frac{I(U,r) H'(U,r)}{H(U,r)}\right)
\\
& = r^5 H(U,r)\bigg\{\frac{1}{r^3} \int_{\Sa_r^+} (ZU)^2\ \Gb y^a + \frac{2}{r^3} \int_{\Sa_r^+} |t| (ZU) F\ \Gb y^a
\\
&\qquad + \frac{2}{r^3} \int_{\Sa_r^+} |t| U F\ \Gb y^a + 2 r \int_{\Rnp\times\{-r^2\}} U F\ \Gb y^a
\\
&\qquad - \frac{4}{r H(U,r)} I(U,r)^2\bigg\}
\\
& = r^2 H(U,r) \int_{\Sa_r^+} (ZU +|t|F)^2\ \Gb y^a -  r^2 H(U,r) \int_{\Sa_r^+} |t|^2 F^2\ \Gb y^a
\\
&\qquad -  2 r^2 H(U,r) \int_{\Sa_r^+} |t| ZU F\ \Gb y^a +  2 r^2 H(U,r) \int_{\Sa_r^+} |t| ZU F\ \Gb y^a
\\
&\qquad + 2 r^2 H(U,r) \int_{\Sa_r^+} |t| U F\ \Gb y^a + 2r^6 H(U,r)\int_{\Rnp\times\{-r^2\}} U F\ \Gb y^a
\\
& \qquad-  \left(\int_{\Sa_r^+} U ZU\ \Gb y^a\right)^2,
\end{align*}
where in the last equality we have used \eqref{IIU}.
We thus obtain
\begin{align*}
\frac{r^5}{2} H(U,r)^2 N'(U,r) & = \int_{\Sa_r^+} U^2\ \Gb y^a \int_{\Sa_r^+} (ZU +|t|F)^2\ \Gb y^a
\\
&\qquad -  \int_{\Sa_r^+} U^2 \ \Gb y^a \int_{\Sa_r^+} |t|^2 F^2\ \Gb y^a
+ 2 \int_{\Sa_r^+} U^2 \ \Gb y^a \int_{\Sa_r^+} |t| U F\ \Gb y^a
\\
&\qquad + 2r^4 \int_{\Sa_r^+} U^2 \ \Gb y^a \int_{\Rnp\times\{-r^2\}} U F\ \Gb y^a
\\
&\qquad - \left(\int_{\Sa_r^+} U ZU\ \Gb y^a\right)^2.
\end{align*}
Cauchy-Schwarz inequality now gives
\[
\left(\int_{\Sa_r^+} U \big(ZU + |t| F\big)\ \Gb y^a\right)^2 \le \int_{\Sa_r^+} U^2\ \Gb y^a \int_{\Sa_r^+} \big(ZU +|t| F\big)\ \Gb y^a.
\]
Substituting in the above we find
\begin{align*}
\frac{r^5}{2} H(U,r)^2 N'(U,r) & \ge \left(\int_{\Sa_r^+} U \big(ZU + |t| F\big)\ \Gb y^a\right)^2 - \left(\int_{\Sa_r^+} U ZU\ \Gb y^a\right)^2
\\
&\qquad -  \int_{\Sa_r^+} U^2 \ \Gb y^a \int_{\Sa_r^+} |t|^2 F^2\ \Gb y^a
+ 2 \int_{\Sa_r^+} U^2 \ \Gb y^a \int_{\Sa_r^+} |t| U F\ \Gb y^a \\
&\qquad + 2r^4 \int_{\Sa_r^+} U^2 \ \Gb y^a \int_{\Rnp\times\{-r^2\}} U F\ \Gb y^a.
\end{align*}
Expanding the first integral in the right-hand side of the latter inequality, and returning to the definitions of $H(U,r)$ and $I(U,r)$, we find
\begin{align*}
\frac{r^5}{2} H(U,r)^2 N'(U,r) & \ge \left(\int_{\Sa_r^+} |t| U F\ \Gb y^a\right)^2 + 4 r^2 I(U,r) \int_{\Sa_r^+} |t| U F\ \Gb y^a
\\
&\qquad -  r^2 H(U,r) \int_{\Sa_r^+} |t|^2 F^2\ \Gb y^a
 + 2 r^2 H(U,r) \int_{\Sa_r^+} |t| U F\ \Gb y^a \\
&\qquad + 2r^6 H(U,r) \int_{\Rnp\times\{-r^2\}} U F\ \Gb y^a.
\end{align*}
This gives
\begin{align*}
N'(U,r) & \ge \frac{4}{r^3}N(U,r) \frac{\int_{\Sa_r^+} |t| U F\ \Gb y^a}{H(U,r)} -  \frac{2}{r^3}  \frac{\int_{\Sa_r^+} |t|^2 F^2\ \Gb y^a}{H(U,r)}
\\
&\qquad +  \frac{4}{r^3}  \frac{\int_{\Sa_r^+} |t| U F\ \Gb y^a}{H(U,r)} + 4r  \frac{\int_{\Rnp\times\{-r^2\}} U F\ \Gb y^a}{H(U,r)}
\\
& =  \frac{4}{r^3}\left(N(U,r)+1\right) \frac{\int_{\Sa_r^+} |t| U F\ \Gb y^a}{H(U,r)} -  \frac{2}{r^3}  \frac{\int_{\Sa_r^+} |t|^2 F^2\ \Gb y^a}{H(U,r)}
\\
&\qquad+ 4r  \frac{\int_{\Rnp\times\{-r^2\}} U F\ \Gb y^a}{H(U,r)}.
\end{align*}
Using the Cauchy-Schwarz inequality, we find
\begin{align}\label{boundN'}
N'(U,r) & \ge -\frac{4}{r^2}\left(N(U,r)+1\right) \frac{\left(\int_{\Sa_r^+} |t|^2F^2\ \Gb y^a\right)^{1/2}}{H(U,r)^{1/2}} -  \frac{2}{r^3}  \frac{\int_{\Sa_r^+} |t|^2 F^2\ \Gb y^a}{H(U,r)}
\\
&\qquad -4r  \frac{\left(\int_{\Rnp\times\{-r^2\}} U^2\ \Gb y^a\right)^{1/2}\left(\int_{\Rnp\times\{-r^2\}}F^2\ \Gb y^a\right)^{1/2}}{H(U,r)}.
\notag
\end{align}
To proceed, we note that for $r\in \Ls$ we have in particular $H(U,r)> 0$, and thus we are in the conditions of Lemma~\ref{L:nondeg}. In particular, we trivially infer from \eqref{record}
\begin{align*}
\left(\int_{\Rnp\times \{-r^2\}} U^2\ \Gb y^{a}\right)^{1/2} &= \left(H(U,r) \left(1 +  N(U,r)\right)\right)^{1/2}\\
&\le H(U,r)^{1/2} \left(1 +  \frac12N(U,r)\right).
\end{align*}
Using this bound in \eqref{boundN'} gives
\begin{align*}
N'(U,r) & \ge -\frac{4}{r^2}\left(N(U,r)+1\right) \frac{\left(\int_{\Sa_r^+} |t|^2F^2\ \Gb y^a\right)^{1/2}}{H(U,r)^{1/2}} -  \frac{2}{r^3}  \frac{\int_{\Sa_r^+} |t|^2 F^2\ \Gb y^a}{H(U,r)}
\\
&\qquad -2r  \frac{(2+N(U,r))\left(\int_{\Rnp\times\{-r^2\}}F^2\ \Gb y^a\right)^{1/2}}{H(U,r)^{1/2}}.
\end{align*}
This estimate implies
\begin{align*}
(2+N(U,r))'&\geq
             -(2+N(U,r))\\
             &\qquad\times \left[\frac{4}{r^2}\frac{\left(\int_{\Sa_r^+}
             |t|^2F^2\ \Gb
             y^a\right)^{1/2}}{H(U,r)^{1/2}} + 2 r \frac{\left(\int_{\Rnp\times\{-r^2\}}F^2\
             \Gb y^a\right)^{1/2}}{H(U,r)^{1/2}}\right]\\
  &\qquad - \frac{2}{r^3}  \frac{\int_{\Sa_r^+} |t|^2 F^2\ \Gb y^a}{H(U,r)}.
\end{align*}
At this point we observe that \eqref{aF}, \eqref{Ghom} and a simple rescaling argument imply
\begin{equation}\label{aF2}
\left(\int_{\Rnp \times\{-r^2\}} F^2\ \Gb y^a dX\right)^{1/2} \le C r^{\ell - 2},
\end{equation}
where $C = \sqrt{C_\ell\ C_{n,a,\ell}}$, with
\[
C_{n,a,\ell} = \int_{\Rnp} |(X,-1)|^{2(\ell-2)} \Gb(X,-1) y^a dX.
\]
Similarly, we obtain
\begin{equation}\label{aF3}
\left(\int_{\Sa_r^+} |t|^2F^2\ \Gb y^a\right)^{1/2} \le C r^{1+ \ell}.
\end{equation}
Now, if $r\in \Ls$ we have $H(U,r)> r^{2\ell - 2 + 2 \sigma}$ and thus from \eqref{aF2}, \eqref{aF3}, and the above estimate for $(2+N(U,r))'$, we find
\begin{align*}
(2+N(U,r))' & \geq - C' (2+N(U,r)) r^{-\sigma} -  C'' r^{1-2\sigma} \ge - \bar C (2+N(U,r)) r^{-\sigma}.
\end{align*}
Since by Lemma~\ref{L:nondeg} we know that $2+ N(U,r) > 0$, this proves \eqref{newidea}, thus completing the proof.
\end{proof}

\section{Gaussian estimates}\label{S:gaussian}

In this section we establish some uniform second derivative estimates in Gaussian spaces that play a crucial role in the blowup analysis in Section~\ref{S:blowups}.

\begin{lemma}\label{Est1}
Let $U \in  \Sp$, with $F, F_t \in L^{\infty}(\Sa_{1}^+)$. Then, for any $0<\rho<1$, there exists a constant $C(n,\rho)>0$ such that the following estimates hold:
\begin{equation}\label{step1}
\int_{\Sa_{\rho}^+}  |t| |\nabla U|^2  \Gb y^a \leq C(n,a,\rho) \int_{\Sa_1^+} (U^2  + |t|^2 F^2) \Gb y^a,
\end{equation}
and
\begin{equation}\label{on1}
\int_{\Sa_{\rho}^+} |t|^2 (|\nabla U_{x_i}|^2 + U_t^2 ) \Gb y^a \leq C(n,a,\rho) \int_{\Sa_1^+} (U^2  + |t|^2F^2) \Gb y^a.
\end{equation}
\end{lemma}

\begin{proof}
We closely  follow the ideas  in Appendix A in \cite{DGPT} and in \cite{BG}. In the rest of the proof, whenever we refer to the weak formulation of \eqref{pareq} we mean that, given $\eta \in W^{1,2}(\Rnp \times (-1, 0), y^adXdt)$ with $\eta(\cdot, t)$ compactly supported in $\B_R^{+}$, for some $R>0$ independent of $t \in (-1, 0)$,  we have for all $0< \delta < r < 1$,
\begin{equation}\label{wk10}
\int_{\Sa_r^+- \Sa_\delta^+} \left( \langle \nabla U, \nabla \eta\rangle   + U_t \eta +F U \eta \right) \Gb y^a=- \int_{S_r-S_\delta} U  \p^a_y  U\  \eta\ \Gb.
\end{equation}
Having clarified this, we divide the proof into four steps.

\medskip
\emph{Step 1:} Let $0<\rho<1$ be fixed. We first establish \eqref{step1}, which represents a Caccioppoli type energy estimate in Gaussian space.
We begin by noting that, since $U$ is in $\Sp$, there exists $R>0$ such that $U(\cdot,t)$ is supported in $\B_R$ for every $t\in (-1,0)$. Let $\tilde \rho$ be such that $\rho < \tilde \rho <1$, and fix $r \in [\rho, \tilde \rho]$. We fix a cut-off function $\hat \tau_0 \in C_0^{\infty}(\Rnn)$ such that $\hat \tau_0 \equiv 0$ outside $\B_R$. Corresponding to such $\hat \tau_0$, for every $k\in \mathbb N$ we define a homogeneous function of degree $k$ in $\Sa_1$ by letting
\begin{equation}\label{tauk}
\tau_k= |t|^{k/2} \hat \tau_0(X/\sqrt{|t|}).
\end{equation}
Using  the test function
 \[
 \eta= U\tau_1^2 \Gb
 \]
in \eqref{wk10}, for $0<\delta<r$ sufficiently small we obtain
\begin{multline}\label{st1}
 \int_{\Sa_r^+- \Sa_\delta^+}  \bigg( |\nabla U|^2 \tau_1^2  + U\left(\langle \nabla U,\frac{X}{2t}\rangle    + U_t\right)\tau_1^2
 \\
  +2U\tau_1 \langle \nabla U, \nabla \tau_1\rangle   + F U\tau_1^2 \bigg) \Gb y^a = - \int_{S_r-S_\delta} U  \p^a_y  U \tau_1^2 \Gb =0.
 \end{multline}
In \eqref{st1} we have used the hypothesis that $U \p^a_y U=0$ on the thin set $\{y=0\}$, see \eqref{gn}, and the reproducing property \eqref{rp}.
Since
 \[
 Z(U^2)= 2U ZU= 2U(\langle X, \nabla U\rangle   + 2t U_t) = 4t U \left(\langle \nabla U,\frac{X}{2t}\rangle   + U_t\right),
 \]
 from \eqref{st1} we have
 \begin{align}\label{int1}
 & \int_{\Sa_r^+- \Sa_\delta^+}  \left( |\nabla U|^2 \tau_1^2   +\frac{1}{4t} Z(U^2)\tau_1^2   +2U\tau_1 \langle \nabla U, \nabla \tau_1\rangle    + F U\tau_1^2 \right) \Gb y^a = 0.
 \end{align}
 Handling the term $Z(U^2)$ in \eqref{int1} requires some care. For this, we argue as on p.92 in the Appendix of \cite{DGPT}, making the change of variables $t = - \la^2$, $X = \la Y$, and exploiting the homogeneity of $\Gb$ and of $\tau_1$. After some work, we find
 \begin{align}\label{int2}
& \int_{\Sa_{r}^{+} - \Sa_{\delta}^+} \frac{1}{4t} Z(U^2)\tau_1^2 \Gb y^a  \geq -r^2 \int_{\Rnp}  U(\cdot, -r^2)^2 \hat \tau_0^2 \Gb( \cdot, -r^2) y^a.
 \end{align}
From \eqref{int1}, \eqref{int2} and Young's inequality, we obtain
\begin{multline}\label{57}
 \int_{\Sa_{r}^{+} - \Sa_{\delta}^+}  |\nabla U|^2 \tau_1^2 \Gb y^a
 \\
 \leq C \left(   \int_{\Rnp} U(\cdot, -r^2)^2 \hat \tau_0^2 \Gb( \cdot, -r^2)  y^a  + \int_{\Sa_r^{+} - \Sa_{\delta}^+} [U^2(\tau_0^4 + |\nabla \tau_1|^2)  + F^2|t|^2]\Gb y^a \right).
 \end{multline}
Integrating  \eqref{57} with respect to $r\in [\rho, \tilde \rho]$, then letting the support of $\hat \tau_0$ sweep $\Rnn$, and $\delta \to 0$, we conclude that the estimate \eqref{step1} holds.

\medskip
 \emph{Step 2:} We turn our attention to the proof of \eqref{on1}. We begin with the following second derivative estimate for tangential derivatives
\begin{equation}\label{cr2}
  \int_{\Sa_{\rho}^+}  |t|^2 |\nabla U_{x_i}|^2 \Gb y^a  \leq C(n, a,\rho) \int_{\Sa_1^+}   (U^2  + |t|^2 F^2) \Gb y^a.
  \end{equation}
With $\rho, \tilde \rho, r, \delta$ as in \emph{Step 1}, for a given $i \in \{1,\ldots, n\}$ and $\ve>0$, we also let
\begin{equation}\label{p100}
 \eta= (|U_{x_i}| - \ve)^{+}\tau_2^2 \Gb,
 \end{equation}
 where $\tau_2$ corresponds to the choice $k = 2$ in \eqref{tauk}.
Noting that the set $ A^{\ve}= \{|U_{x_i}|>\ve \} \cap \{y=0\}$ is compactly contained in the interior of the set $\{\p^a_y  U=0\}$, a standard difference quotient argument as in \cite[Section 5]{BG} allows us to assert that $\nabla U_{ x_i},  \partial_t U_{x_i} \in L^{2}_{\rm loc}(\cdot, y^a dXdt)$ up to $\{y=0\}$ in $A_\ve$ (we stress that here we crucially use the fact that $F_t \in L^{\infty}(\Sa_{1}^+)$). Once we know this, with $\eta$ as in \eqref{p100}, we use $\eta_{x_i}$ as a test function in the weak formulation \eqref{wk10}. Integrating by parts with respect to $x_i$ and by a limiting type argument (i.e.,\ by first integrating in the region $\{y>\beta\}$, and then letting $\beta \to 0$), we  obtain
\begin{align}\label{nt10}
&
\int_{ (\Sa_{r}^{+} - \Sa_{\delta}^+) \cap B^{\ve}} \bigg( |\nabla U_{x_i}|^2 \tau_2^2 +   V_\ve\langle \nabla U_{x_i},\frac{X}{2t}\rangle   \tau_2^2 + \partial_t U_{x_i} V_\ve \tau_2^2 +2 V_\ve \tau_2 \langle \nabla U_{x_i}, \nabla \tau_2\rangle  \bigg) \Gb y^a
\\
&\qquad + \int_{ (\Sa_{r}^{+} - \Sa_{\delta}^+) \cap B^\ve} y^a \left( F (V_{\ve})_{ x_i} \tau_2^2\Gb + 2F V_{\ve} \tau_2 (\tau_2)_{x_i} \Gb + FV_{\ve}\tau_2^2(\Gb)_{x_i}   \right)
\notag
\\
&\qquad - \int_{(S_r - S_\delta) \cap A^{\ve} } (\p^a_y  U)_{x_i} ( |U_{x_i}|-\ve)^{+} \tau_2^2 \Gb =0,
\notag
 \end{align}
 where $V_\ve= (|U_{x_i}| - \ve)^{+}$ and   $B^{\ve}= \{(X,t)\mid |U_{x_i}|>\ve \} $, and in the second term in the first integral in the left-hand side we have used \eqref{rp}. We stress that $\eta_{x_i}$ is not a legitimate test function. Nevertheless, the computation in \eqref{nt10} can be justified by using as a test function  difference quotients of the form
  \[
 \eta_{h, i}=\frac{\eta(X+he_i)-\eta(X)}{h},
 \]
 instead of $\eta_{x_i}$,
and then finally let $h\to0$. We also note that for the difference quotients the integration by parts with respect to $x_i$ is equivalently replaced by an identity of the following type
 \[
 \int_{\Rnp}  f_{h, i}\  g y^a dX= -\int_{\Rnp} f\ g_{-h, i} y^a dX,
 \]
 which holds for arbitrary compactly supported functions $f,g$ and is a consequence of a standard change of variable formula.
In \eqref{nt10}, we have also used that, since $\p_y^a U=0$ on the set $A_\ve$, we have $(\p_y^a U)_{x_i}=0$ on $A_\ve$.   Letting $\ve \to 0$ in \eqref{nt10}, we find
 \begin{multline}\label{bddd}
 \int_{ \Sa_{r}^{+} - \Sa_{\delta}^+}   \left( |\nabla U_{x_i}|^2 \tau_2^2  + \frac{1}{4t} Z (U_{x_i}^2) \tau_2^2  + 2U_{x_i} \langle \nabla U_{x_i}, \nabla \tau_2\rangle   \tau_2 \right) \Gb y^a
 \\\leq \int_{\Sa_r^+ -\Sa_\delta^+} \left( |F| |U_{x_ix_i}| \tau_2^2 + 2F|U_{x_i}| \tau_2|\nabla \tau_2|  + F|U_{x_i}|\tau_2^2
 \left|\frac{X}{2t}\right| \right)  \Gb y^a.
  \end{multline}
To handle the term with $Z (U_{x_i}^2)$ we argue again as in the opening of page 92 in \cite{DGPT}, obtaining
 \begin{align}\label{int5}
& \int_{\Sa_{r}^{+} - \Sa_{\delta}^+} \frac{1}{4t} Z(U_{x_i}^2)\tau_2^2 \Gb y^a \geq -r^2 \int_{\Rnp}  U_{x_i}(\cdot, -r^2)^2 \tau_1^2 \Gb(\cdot, -r^2) y^a.
 \end{align}
The  integral
 \[
 \int_{\Sa_{r}^{+} - \Sa_{\delta}^+}   F|U_{x_i}|\tau_2^2 \left|\frac{X}{2t} \right|\Gb y^a
 \]
in the right-hand side of \eqref{bddd}
can be estimated by Young's inequality as follows:
\begin{multline}\label{bdd1}
\int_{\Sa_{r}^{+} - \Sa_{\delta}^+}   F|U_{x_i}|\tau_2^2\left|\frac{X}{2t} \right|\Gb y^a\\ \leq  C_1 \int_{\Sa_{r}^{+} - \Sa_{\delta}^+}  |F|^2 \tau_2^2 \Gb y^a
+ C_2 \int_{\Sa_{r}^{+} - \Sa_{\delta}^+}   |U_{x_i}|^2 \frac{|X|^2}{2|t|} \tau_1^2 \Gb y^a
\\
\leq C_3 \int_{\Sa_{r}^{+} - \Sa_{\delta}^+}  |F|^2 \tau_2^2 \Gb y^a + C_4 \int_{\Sa_{r}^{+} - \Sa_{\delta}^+}  \left( |\nabla U|^2 (|\nabla \tau_2|^2 + |\tau_1|^2) + |\nabla U_{x_i}|^2 \tau_2^2 \right)\Gb y^a.
\end{multline}
In the last step we have used the fact that the  following inequality  holds  at every time level: for any $v \in W^{1,2}(\R^{n+1}, \Gb y^a dX)$ one has
\begin{equation}\label{int50}
\int_{\Rnp}   v^2 \frac{|X|^2 }{|t|} \Gb y^a \leq  C \int_{\Rnp}   (v^2 + |t| |\nabla v|^2)\Gb y^a.
\end{equation}
We mention that \eqref{int50} corresponds to the inequality (8.17) in \cite{BG}.

The remaining integrals in the right-hand side of \eqref{bddd} can be estimated in a similar way. Using Young's inequality we find
\begin{multline}\label{int52}
\int_{\Sa_{r}^{+} - \Sa_{\delta}^+}  \left( |F| |U_{x_ix_i}| \tau_2^2 + 2F|U_{x_i}| \tau_2|\nabla \tau_2|   \right) \Gb y^a
\\
\leq  \int_{\Sa_{r}^{+} - \Sa_{\delta}^+}  \left(  \frac{1}{4} |\nabla   U_{x_i}|^2 \tau_2^2   + C |F|^2 \tau_2^2   +  |\nabla U|^2 |\nabla \tau_2|^2 \right) \Gb y^a.
\end{multline}
Combining the estimates  \eqref{bdd1}--\eqref{int52}, and subtracting from the left-hand side the integral
\[
\frac{1}{4}  \int_{\Sa_{r}^{+} - \Sa_{\delta}^+} |\nabla U_{x_i} |^2 \tau_2^2 \Gb y^a,
\]
we obtain
\begin{multline}
\int_{ \Sa_{r}^{+} - \Sa_{\delta}^+}    |\nabla U_{x_i}|^2 \tau_2^2 \Gb y^a \leq C \bigg( \int_{\Rnp}  U_{x_i}(\cdot, -r^2)^2 \tau_1^2 \Gb(\cdot, -r^2) y^a
 \\
+   \int_{\Sa_{r}^{+} - \Sa_{\delta}^+}  \left( |F|^2 \tau_2^2  + |\nabla U|^2 (|\nabla \tau_2|^2 +\tau_1^2) \right)\Gb y^a\bigg).
  \end{multline}
As before, we now integrate over $r \in [\rho, \tilde \rho]$, let the support of $\hat \tau_0$ exhaust the whole of $\Rnn$, then let $\delta\to 0$, and also using the previously established estimate \eqref{step1}, we finally deduce that \eqref{cr2} holds.

\medskip
\emph{Step 3:} Our next objective is to establish the following  second derivative estimate in the normal direction:
\begin{align}\label{cr0}
 & \int_{\Sa_{\rho}^+} |t|^2 ( (y^a U_y)_y)^2 \Gb y^{-a}
  \leq  C(n, a,\rho) \int_{\Sa_1^+} ( U^2 + |t|^2 F^2)\Gb y^a.
 \end{align}
For this we make use of the following conjugate equation which is satisfied in $\Rnp \times (-1, 0)$ by  $w= y^a U_y$
\begin{equation}\label{conj}
  \operatorname{div}(y^{-a} \nabla w) - y^{-a} w_t= F_y.
  \end{equation}
 For a given $\ve>0$, we  consider the test function
 \[
   \eta=(y^a U_y - \ve)^{+} \tau_2^2 \Gb,
 \]
 in the weak formulation of \eqref{conj}. We note that since $\p^a_y U \leq 0$ on the thin set $\{y=0\}$, thanks to the H\"older continuity of $y^a U_y$ up to $\{y=0\}$, the function $\eta$ is compactly supported in the region $\{y>0\}$, and therefore it is a legitimate test function. With $w=y^a U_y$, we thus have
 \begin{multline}\label{st5}
\int_{\Sa_r^+ - \Sa_\delta^+}\Big(  |\nabla (w-\ve)^{+}|^2 \tau_2^2 \Gb y^{-a}  +  \tau_2^2 (w-\ve)^{+}
 \langle \nabla w,\nabla \Gb\rangle   y^{-a}\\
 + (w-\ve)^+ w_t\tau_2^2 \Gb y^{-a}
+2 \tau_2 (w-\ve)^+ \langle \nabla w, \nabla \tau_2\rangle   \Gb y^{-a}\\
+  F \partial_y  (w-\ve)^{+}\tau_2^2\Gb + 2 F  (w-\ve)^{+}\tau_2 \partial_y \tau_2 \Gb
+ F (w-\ve)^{+} \tau_2^2 \frac{y}{2t} \Gb\Big) =0.
 \end{multline}
Similarly to \eqref{int5}, we now obtain
\begin{align}\label{w10}
 & \int_{\Sa_r^+ - \Sa_\delta^+}  [\tau_2^2 (w-\ve)^{+}  \langle \nabla w,\nabla \Gb\rangle   y^{-a}  + (w-\ve)^+ w_t\tau_2^2\Gb]
 \\
 &\qquad = \int_{\Sa_r^+- \Sa_\delta^+}   \frac{1}{4t} Z(((w-\ve)^{+})^2) \tau_2^2 \Gb y^{-a}
 \notag\\
 &\qquad \geq -r^2 \int_{\Rnp}  (w-\ve)_{+}^{2}(\cdot, -r^2) \tau_1^2 \Gb( \cdot, -r^2) y^{-a}. \notag
 \end{align}
Using Young's inequality, we estimate
\begin{align*}
& \int_{\Sa_r^+ - \Sa_\delta^+}  y^{-a} 2\tau_2 (w-\ve)^+ \langle \nabla w, \nabla \tau_2\rangle   \Gb +  F \partial_y  (w-\ve)^{+}\tau_2^2\Gb + 2F   (w-\ve)^{+}\tau_2 \partial_y \tau_2 \Gb
   \\
   &\qquad \leq  \int_{\Sa_r^+ - \Sa_\delta^+}  \frac{1}{4} \tau_2^2 |\nabla (w-\ve)^+|^2 \Gb y^{-a} + C(  ((w-\ve)^+)^2 |\nabla \tau_2|^2 \Gb y^{-a} + F^2 \tau_2^2 \Gb y^{-a}).
   \notag
\end{align*}
Finally, the last term in the left-hand side of \eqref{st5} is estimated in the following way. First, Young's inequality (and the trivial observation that $|y|\le |X|$) gives
\begin{multline}\label{ie}
   \int_{\Sa_r^+ - \Sa_\delta^+} F  (w - \ve)^{+} \tau_2^2 \frac{y}{2t} \Gb\\ \leq c_0^{-1}   \int_{\Sa_r^+ - \Sa_\delta^+}t  F^2 \tau_2^2 \Gb y^a + c_0 \int_{\Sa_r^+ - \Sa_\delta^+}  ((w-\ve)^{+})^2  \tau_1^2 \frac{|X|^2}{2t} \Gb y^{-a}),
\end{multline}
 where $c_0$  can be chosen arbitrarily small. To control the second integral in the right-hand side of \eqref{ie} we argue similarly to \eqref{int50}, but with $a$ replaced by $-a$. Inserting the ensuing estimate in \eqref{ie}, the resulting inequality becomes
 \begin{multline}\label{et00}
\int_{\Sa_r^+ - \Sa_\delta^+}  F  (w - \ve)^{+} \tau_2^2 \frac{y}{2t} \Gb \\\leq C_1  \int_{\Sa_r^+ - \Sa_\delta^+}  F^2 \tau_2^2 \Gb y^{a} +  C_1 c_0 \int_{\Sa_r^+ - \Sa_\delta^+}  \bigg(|\nabla (w-\ve)^{+}|^2 \tau_2^2 \Gb y^{-a}
\\
+ ((w-\ve)^{+})^2 |\nabla \tau_1|^2 |t| \Gb y^{-a} + ((w-\ve)^{+})^2  \tau_1^2 \Gb y^{-a}\bigg),
 \end{multline}
for some constant $C_1>0$. At this point we choose $c_0$ small enough such that $C_1 c_0 < \frac{1}{4}$.
Combining the estimates \eqref{w10} and \eqref{et00}, and then subtracting the integral
 \[
\frac{1}{2} \int_{\Sa_{r}^+ - \Sa_{\delta}^+}   |\nabla (w-\ve)^{+}|^2 \tau_2^2 \Gb y^{-a}
 \]
 from the left hand side of \eqref{st5},
 we finally obtain
  \begin{multline*}
 \int_{\Sa_{r}^+ - \Sa_{\delta}^+}  |t|^2 |\nabla (w-\ve)^+|^2 \Gb y^{-a} \\ \leq  C  \bigg\{\int_{\Sa_{r}^{+} - \Sa_{\delta}^{+}} \bigg[((w-\ve)^{+})^2  ( \tau_1^2  + |t| |\nabla \tau_1|^2 + |\nabla \tau_2|^2) \Gb y^{-a}\bigg]
 +   |t|^2 F^2 \Gb y^a\\ + r^2 \int_{\Rnp}  (w-\ve)_{+}^{2}(\cdot, -r^2) \tau_1^2 \Gb( \cdot, -r^2) y^{-a}\bigg\}.
\end{multline*}
Integrating again in $r \in [\rho, \tilde \rho]$, by a limiting argument and finally letting $\ve \to 0$, we deduce the following estimate
\begin{equation}\label{cr5}
 \int_{\Sa_{\rho}^+} |t|^2 |\nabla w^+|^2 \Gb y^{-a}  \leq C(n, \rho) \int_{\Sa_{\tilde \rho}^+} |t| w^2 \Gb y^{-a}  +  |t|^2 F^2\Gb y^a,
 \end{equation}
 where $\rho < \tilde \rho < 1$.

As a next step, we obtain an estimate similar to \eqref{cr5} for $\nabla w^{-}$ in $L^{2}(\cdot,\Gb y^{-a}  dXdt)$. Given $\ve>0$,  we consider the function $v_\ve= (w+\ve)^{-}$, where $w=y^a U_y$ as before. Let $C^\ve= \operatorname{supp} v_\ve$, and denote by $\tilde C^{\ve}$ the reflected portion of $C^{\ve}$ across $\{y=0\}$. Since $U \p^a_y U =0$ on $\{y=0\}$, by the continuity of $y^a U_y$ up to $\{y=0\}$, we see that $C^{\ve}\cap  \{y=0\}$ is contained in the interior of $\{(x, 0)\mid U(x, 0)=0\}$. Therefore near $C^{\ve} \cap \{y=0\}$, if $U$ is oddly reflected, then the extended $U$ solves the following equation in $D^\ve=C^{\ve} \cup \tilde C^{\ve}$
 \[
 \operatorname{div}(|y|^a \nabla U)- |y|^a U_t=  |y|^a \tilde F,
 \]
 where $\tilde F$ is the odd extension of $F$ across $\{y=0\}$.  Therefore, in the set $D^{\ve}$ the function $v=|y|^a U_y$ is an  even extension of $w$ across $\{y=0\}$, and it solves the following conjugate equation in $D^{\ve}$
 \begin{equation}\label{cong10}
 \operatorname{div}(|y|^{-a} \nabla v)- |y|^{-a} v_t= \tilde F_y.
 \end{equation}
Using the test function
\[
\eta= (w+\ve)^{-} \tau_2^2 \Gb
\]
in the weak formulation of \eqref{cong10}, arguing as in \eqref{st5}--\eqref{cr5}, and finally letting $\ve \to 0$, we obtain the following estimate
\begin{equation}\label{cr4}
 \int_{\Sa_{\rho}^{+}} |t|^2 |\nabla w^-|^2 \Gb y^{-a} \leq C(n, \rho) \int_{\Sa_{1}^+} |t| w^2 \Gb y^{-a}  +  |t|^2F^2\Gb y^a.
 \end{equation}
By combining \eqref{cr5} and \eqref{cr4}, and using the previously established estimate \eqref{step1}, we finally have
 \begin{align}
 & \int_{\Sa_{\rho}^+} |t|^2  ( (y^a U_y)_y)^2 \Gb y^{-a} \leq \int_{\Sa_{\rho}^+}   |t|^2 |\nabla w|^2 \Gb y^{-a}
 \leq  C(n, \rho) \int_{\Sa_1^+} ( U^2 + |t|^2 F^2)\Gb y^a,
 \notag
 \end{align}
 which completes the proof of \eqref{cr0}.

\medskip
\emph{Step 4:} At this point, using the equation satisfied by $U$, the corresponding estimate for $U_t$ claimed in \eqref{on1} follows from  \eqref{cr2} and  \eqref{cr0}. This finishes the proof of the lemma.
\end{proof}

\begin{rmrk}
We note that assumption that $\|F_t\|_{L^{\infty}(\Sa_{1}^+)}$ be finite in Lemma~\ref{Est1} is not restrictive since it can always be achieved if the obstacle is sufficiently regular.
\end{rmrk}

We also need the following estimate in our blowup analysis in Section~\ref{S:blowups}.

\begin{lemma}\label{Est2}
For $i=1,2$,  let $U_i \in  \mathfrak S_{F_i}(\Sa_1^+)$ with $F_i \in L^{\infty}(\Sa_{1}^+)$. Then, for any $0<\rho<1$, one has
\begin{equation}\label{nd1}
\int_{\Sa_{\rho}^+}  |t| |\nabla (U_1 - U_2)|^2  \Gb y^a \leq C(n, a, \rho ) \int_{\Sa_1^+}   (U^2  + |t|^2 (F_1-F_2)^2) \Gb y^a
\end{equation}
\end{lemma}

\begin{proof}
First, we even reflect  $U_1, F_1,U_2, F_2$ across $\{y=0\}$.  Consider now  $U=U_1-U_2$. We claim that the following holds in $\Sa_1$,
\begin{equation}\label{cl10}
\La U^{+} \leq |y|^a (F_1- F_2)^{+},\quad \La U^{-} \leq |y|^a (F_1- F_2)^{-}.
\end{equation}
We note that  it is clear that the differential inequalities \eqref{cl10} are respectively satisfied in $\Sa_1^{+}$ and $\Sa_1^-$. Therefore, we only need to show the inequality near a point $(x_0,0,t_0) \in Q_{3/4}$. Suppose $U(x_0,0,t_0) >0$. Since it must necessarily be $U_1(x_0, 0, t_0) >0$, we infer the existence of a sufficiently small $\delta>0$, such that $\p_y^a U_1 \equiv 0$ in $Q_{\delta}(x_0,0,t_0)$.  This implies that
\[
\La U_1 = |y|^a F_1
\]
in $\Q_{\delta}(x_0,0,t_0)$.   On the other hand,
\[
\La U_2 \geq |y|^a F_2
\]
Therefore we obtain
\[
\La U \leq |y|^a (F_1 - F_2)
\]
in $\Q_{\delta}(x_0, 0, t_0)$.  Thus
\[
\La U \leq |y|^a (F_1 - F_2)^{+}
\]
in $\Q_{3/4} \cap \{U>0\}$. Now  by using  a standard argument as in  the proof of Lemma~2.1 in \cite{PZ},  we can  deduce  that
\[
\La U^+ \leq |y|^a (F_1- F_2)^+
\]
in $\Q_{3/4}$ and hence in $\Sa_1$. The argument for $U^{-}$ is similar and thus we can assert that  \eqref{cl10} holds. Now  given the validity of \eqref{cl10}, we can argue  as  in {Step 1}  in the proof of Lemma~\ref{Est1}    (using $\eta= U^{\pm} \tau_1^2 \Gb$ as a test function in the weak formulation for $U^{\pm}$) to conclude that  the  weighted Caccioppoli type estimate \eqref{nd1} holds.
\end{proof}

\section{Existence and homogeneity of blowups}\label{S:blowups}

Throughout this section, we assume that  $U \in \Sp$, where $F$ satisfies \eqref{aF} for some $\ell\ge 2$. Towards the end of the section, we will need $\ell\geq 4$ and require the following additional bounds to hold for and some positive constant $C_\ell$
\begin{alignat}{2}
\label{aF-2}|\nabla_X F(X,t)| &\le C_\ell |(X,t)|^{\ell -3},&\quad&\text{for}\ (X,t)\in \Q_{1/2}^+\\
\label{aF-3}|\p_t F(X,t)| &\le C_\ell |(X,t)|^{\ell -4} && \text{for}\ (X,t)\in \Q_{1/2}^+.
\end{alignat}
We note that \eqref{aF-2}, \eqref{aF-3} are fulfilled by assuming that the obstacle be of class $H^{\ell,\ell/2}$, see \eqref{nablafbound}, \eqref{ptfbound} at the end of Section~\ref{S:classes}. We now state our  first result.

\begin{lemma}\label{L:2plus2sigma}
Under assumptions of Theorem~\ref{T:poon}, fix $\sigma \in (0, 1)$. Then, one has
\begin{equation}\label{Phi1}
\Phi_{\ell, \sigma}(U,0^+)  = \kappa \le  \ell - 1 +  \sigma.
\end{equation}
Furthermore, if $\kappa<  \ell - 1 + \sigma$, then there exists $r_0 = r_0(U)>0$ such that for every $r\in (0,r_0)$ one has
\begin{equation}\label{Phi2}
H(U,r) \ge r^{2\ell -2 +2\sigma}.
\end{equation}
 In particular, we have in such case
\begin{equation}\label{Phi3}
 \Phi_{\ell, \sigma}(U,0^+) =\frac{1}{2} \underset{r\to 0^+}{\lim} \frac{r H'(U,r)}{H(U,r)} = \underset{r\to 0^+}{\lim} N(U,r) = \underset{r\to 0^+}{\lim}  \tilde N(U,r).
\end{equation}
\end{lemma}

\begin{proof}
The proof of \eqref{Phi1} and \eqref{Phi2} follows the lines of that of Lemma~7.1 in \cite{DGPT} in the case $a = 0$, and we thus refer the reader to that source for details. In order to establish \eqref{Phi3} we note explicitly that from \eqref{IU}, \eqref{fU} and Lemma~\ref{L:IIU}, we have
\[
N(U,r) = \tilde N(U,r) - \frac{\frac{2}{r^2} \int_{\Sa_r^+} |t| U F\ \Gb y^{a}}{H(U,r)}.
\]
Keeping in mind that the hypothesis \eqref{aF} implies \eqref{aF3}, Cauchy-Schwarz's inequality gives
\begin{align*}
\frac{\left|\frac{1}{r^2} \int_{\Sa_r^+} |t| U F\ \Gb y^{a}\right|}{H(U,r)} &\le \frac{1}{r} \frac{\left(\int_{\Sa_r^+} |t|^2 F^2\ \Gb y^{a}\right)^{1/2}}{H(U,r)^{1/2}} \\&\le C \frac{r^{\ell}}{H(U,r)^{1/2}}  \le C \frac{r^{\ell}}{r^{\ell -1 +\sigma}} = C r^{1-\sigma}\to 0,
\end{align*}
as $r\to 0^+$. This shows that
\[
\underset{r\to 0^+}{\lim} N(U,r) = \underset{r\to 0^+}{\lim} \tilde N(U,r).\qedhere
\]
\end{proof}

Later in the paper we will need to work with two different families of rescalings, which we now introduce.

\begin{dfn}\label{D:Ahomrescalings}
With $\delta_\la$ as in \eqref{pardil} we define the \emph{parabolic Almgren rescalings} of $U$ as
\begin{equation}\label{Ar}
U_r = \frac{U\circ \delta_r}{H(U,r)^{1/2}}.
\end{equation}
For $\kappa>0$ we define the $\kappa$-\emph{homogeneous} rescalings of $U$ as
\begin{equation}\label{homr}
\tilde U_r =  \frac{U\circ \delta_r}{r^\kappa}.
\end{equation}
\end{dfn}
We note that the  rescaled functions $U_r$ solve
\begin{equation}\label{g1}
 \begin{cases}
\La U_r = |y|^a  F_r&\text{in}\ \Sa_1^+ \cup \Sa_1^-,
\\
\min\{U_r(x,0,t), -\partial_y^a U(x,0,t)\}=0 & \text{on }S_1,\\
U_r(x,-y,t) = U_r(x,y,t),&\text{in}\ \Sa_1,
\end{cases}
\end{equation}
where
\begin{equation}\label{Fr}
F_r= \frac{r^2F(rX, r^2t)}{ H(U, r)^{1/2}}.
\end{equation}
We have the following key result, whose elementary verification we leave to the reader.

\begin{prop}\label{P:rescaledU}
For every $r, \rho >0$ one has
\begin{equation}\label{rescaleU}
H(U,r \rho) = H(U\circ \delta_r,\rho),\quad D(U,r \rho) = D(U\circ \delta_r,\rho),\quad N(U,r\rho) = N(U\circ \delta_r,\rho).
\end{equation}
In particular, we have for the parabolic Almgren rescalings
\begin{equation}\label{ArescaleU}
N(U_r,\rho) = N(U,r \rho).
\end{equation}
\end{prop}

The Almgren rescalings are tailor-made for Theorem~\ref{T:poon}, whereas the homogeneous rescalings are the appropriate ones for the applications of the Weiss and Monneau type monotonicity formulas in Theorems \ref{T:weisspar} and \ref{T:monneau} below. Proposition~\ref{P:rescaledU} implies in particular that
\[
H(U_r, 1)=1,
\]
and,
more generally,
\[
H(U_r,\rho)= \frac{H(U, \rho r)}{H(U,r)}.
\]
 The following lemma plays a key role  in our blowup analysis. It will ensure that the blow up limit $U_0$ is bounded on sets of the type $\B_{A}^+ \times (-1, 0]$ for any $A>0$.

\begin{lemma}\label{hr}
Let $U \in \Sp$, where $F$ satisfies \eqref{aF} for some $\ell \geq 2$. Given $A>0$ and  $0<r < 1$, we have
\begin{equation}\label{des}
\|U\|_{L^{\infty}(\B_{Ar}^{+} \times (-r^2/4, 0))} \leq C H(U,r)^{1/2}  + C r^{\ell},
\end{equation}
for some universal $C>0$ depending also on $\ell, A$.
\end{lemma}

\begin{proof}
It suffices to prove the claim for $U^{+}$  and $U^{-}$ since, after even reflection in $y$, both of them satisfy for any $r>0$,
\[
\La V \leq |y|^a \|F\|_{L^{\infty}(\Q_r)}\quad\text{in}\ \Q_r.
\]
Let then $V$ denote either $U^+$ or $U^-$. Since $F$ satisfies \eqref{aF} with some constant $C_\ell$, we let
\[
\tilde V= \begin{cases}V + C_0 ( |X|^2 - t)^{\ell/2}, &\text{when}\ \ell >2,
\\
V - C_0\ t, &\text{when}\ \ell=2,
\end{cases}
\]
where $C_0 = C_0(n,\gamma, a,\ell)$ is so chosen that
$\La \tilde V \leq 0$.
A simple calculation shows that this can be ensured. We then note that  the following holds
\[
H(\tilde V, r)^{1/2}  \leq H(U, r) ^{1/2} + C_0 \left( \frac{1}{r^2} \int_{\Sa_r^+}  (|X|^2 + |t|)^{\ell}\  \Gb y^a\right)^{1/2}  \leq H(U, r) ^{1/2} + C r^{\ell},
\]
where in the last inequality we have used a change of variable and the  homogeneity property of $\Gb$. Now, given $(X_0,t_0) = (x_0, y_0, t_0) \in \B_{Ar}^{+} \times (-r^2/4, 0)$, and using that $\tilde V$ has a polynomial growth at infinity, we can adapt a variational argument in \cite{DG} to deduce the following sub-mean value estimate
\begin{equation}\label{tc}
\tilde V(X_0,t_0) \leq \int_{\Rnp} V(x,y, t_0 +s_0) \Gb(X,X_0,s_0) y^a  dX,
\end{equation}
where $\Gb(X,X_0,s_0) = \mathscr G_a(X,X_0,-s_0)$, see \eqref{sGa} and \eqref{Gbarra}, and $s_0  < t_0 <0$.

Indeed, \eqref{tc} can be justified as follows.   Let $\tau_R$ be a cut-off of the type  $\tau_R (X)= \tau^{1} _R (x) \tau^{2} _R(y)$, with $\tau^{1}_R \equiv 1$ in $B_R$ and vanishing outside $B_{2R}$, and $\tau^{2}_R \equiv 1$ in ($-R, R)$ and vanishing outside $(-2R, 2R)$. Then, for $t <0$ the function  $w= \tilde V \tau_R$ solves
\begin{equation}\label{w}
\La w \leq C y^a ( |\nabla \tilde V | |\nabla \tau_R | +  |\tilde V| ( |\nabla^2 \tau_R| + |\nabla \tau_R|)),
\end{equation}
for some universal $C$ independent of $R$, provided $R>1$. Note that it is not restrictive to assume $R$ large, since we eventually want to let $R \to \infty$.
Fix $t_0 <0$, and for a given $s_0 < 0$, define for $s \leq \sqrt{-s_0}$,
\[
\phi(s)= \int_{\Rnp}  w(X, t_0+ s_0+ s^2) \mathscr G_a(X,X_0,-s_0 -s^2) y^a dX.
\]
Since  $w$ is Lipschitz in $s$, $\phi(s)$ is absolutely continuous. Therefore, differentiating under the integral sign, one has for a.e. $s$,
\begin{align}\label{e109}
& \phi'(s) = 2 s \int_{\Rnp}  \left(\partial_t w( X, t_0 +s_0 + s^2) \mathscr G_a(X, X_0,-s_0 -s^2)  -  w \partial_t \mathscr G_a(X, X_0,-s_0 -s^2)\right)y^a.
\end{align}
We now integrate by parts in the second integral in \eqref{e109}. Using the properties
\begin{equation}\label{610}
\begin{cases}
y^a \partial_t \mathscr G_a = \operatorname{div}(y^a \nabla \mathscr G_a),
\\
\partial_y^a\mathscr G_a=0\text{ on $y=0$},
\end{cases}
\end{equation}
 and \eqref{w}, we deduce
\[
\phi'(s) \leq   2 s \int_{\Rnp}  G_R\  \mathscr G_a(X, X_0,-s_0 -s^2) y^a,
\]
where we have let $G_R= 2(| \nabla \tilde V | |\nabla \tau_R | + |\tilde V| |\nabla^2 \tau_R|)$. Note that $G_R$ is supported in the region where $|X|  \geq R$.
We next integrate the latter inequality on the interval $(0,t)$, finding
\begin{multline}\label{fs2}
\int_{\Rnp}  w(X, t_0+ s_0+t^2)\mathscr G_a(X, X_0,-s_0 -t^2) y^a \\
\leq  \int_{\Rnp}  w(X, t_0+ s_0) \mathscr G_a(X, X_0,-s_0) y^a
\\
+ \int_{0}^{t^2}   \int_{\Rnp}  G_R ( X,t_0 +s_0  + s )\mathscr G_a(X, X_0,-s_0 -s) y^a ds.
\end{multline}
At this point we let  $t \to \sqrt{-s_0}$ in \eqref{fs2}. Using the Dirac property of $\mathscr G_a$ and changing the time variable from $t_0+s_0+s$ to $s$, we obtain
\begin{align}\label{fund}
w(X_0, t_0) & \leq \int_{\Rnp} w(X, t_0+ s_0) \mathscr G_a(X, X_0,-s_0) y^a
\\
& +  \int_{t_0+s _0}^{t_0}   \int_{\Rnp} G_R(X,s)  \mathscr G_a(X, X_0,t_0-s) y^a dX ds.
\notag\end{align}

To proceed further, with $t_0, s_0$ as in \eqref{fund} above, we now  fix $b>0$ small enough  such that $-1/2 < t_0 < -b$ and $s_0< -b$. Given $A>0$, let $X_0$ be such that  $|X_0|\leq A$. Since we eventually want to let $R\to \infty$ in \eqref{fund}, we assume that $R$ be sufficiently large, say $R \geq 100 A +1$. We make the following:
\begin{claim}\label{claim1} For $-1/2<t_0<-b<0$ and $|X_0|\leq A$, there is $C=C(n,b,A)>0$ such that for $R\geq 100A+1$ we have
\begin{multline*}
{\mathscr G_a(X, X_0,t_0-s) \leq
\left\{\begin{alignedat}{3}
&C\,\Gb(X,s),
&\quad&\text{if $s<t_0$,  $t_0-s< -s/8$, $|X| \geq R$,}
&\qquad& \textup{(a)}\\
&C\,\Gb(X,s) e^{C|X|},
&\quad&\text{if $s< t_0$, $t_0-s > -s/8$}.&\qquad&\textup{(b)}
\end{alignedat}
\right.
}
\hspace{-\multlinegap}
\end{multline*}
\end{claim}
To establish the claim, we recall that $\mathscr G_a(X, X_0,t_0-s) = p(x, x_0,t_0-s ) p^{(a)}(y,y_0,  t_0-s)$. Keeping in mind the representation \eqref{fs} of $p^{(a)}$, and the asymptotic behavior of the modified Bessel function $I_{\frac{a-1}{2}}$ (see e.g. (5.11.10) and (5.16.14) in \cite{Le}), we note that for any $a\in (-1,1)$, there exist universal constants $C(a), c(a) >0$, such that
\begin{equation}\label{b1}
\begin{cases}
I_{\frac{a-1}{2}}(z) \leq C(a) z^{-1/2} e^{z},& z>  c(a),
\\
I_{\frac{a-1}{2}}(z) \leq  C(a) z^{\frac{a-1}{2}},& 0<z <c(a).
\end{cases}
\end{equation}
Moreover, it is easy to check that for $|X| \geq 100 A +1$ and $|X_0| \leq A$ one has
\begin{equation}\label{b2}
|X-X_0|^2 \geq \frac{1}{2} + 2yy_0+  \frac{ |X|^2}{8}.
\end{equation}
We also note from \eqref{fs} that
\begin{multline}\label{b3}
p^{(a)}(y,y_0,  t_0-s) = e^{-\frac{yy_0}{2(t_0-s)}}(2(t_0-s))^{-\frac{a+1}{2}} \left(\frac{yy_0}{2(t_0-s)}\right)^{-\frac{a-1}{2}} \\\times I_{\frac{a-1}{2}}\left(\frac{yy_0}{2(t_0-s)}\right) e^{-\frac{(y-y_0)^2}{4(t_0-s)}}.
\end{multline}
We now subdivide the proof of both (a) and (b) in Claim \ref{claim1} into two cases: 1) $\frac{yy_0}{2(t_0-s)} < c(a)$; and, 2) $\frac{y y_0}{2(t_0-s)} \geq c(a)$.
In case 1) we have from
the second inequality in \eqref{b1}: $\left(\frac{yy_0}{2(t_0-s)}\right)^{-\frac{a-1}{2}}  I_{\frac{a-1}{2}}\le C(a)$.
Substituting this information in \eqref{b3}, using \eqref{b2} and the fact that $s < -b$, we find
\[
\mathscr G_a(X, X_0,t_0-s) \leq   \frac{C}{(t_0-s)^{\frac{n+1+a}{2}} }e^{-\frac{|X-X_0|^2}{4(t_0-s)}} \leq  \frac{C(n, b, A)}{(-s)^{\frac{n+1+a}{2}}(t_0-s)^{\frac{n+1+a}{2}}} e^{-\frac{1}{8(t_0-s)} }  e^{- \frac{|X|^2}{32(t_0-s)}}.
\]
Note that in the second inequality above, we have also used the fact that
 $e^{-\frac{yy_0}{2(t_0-s)}} \leq 1$.
From the latter estimate the desired bound in case (a) of Claim \ref{claim1} follows using that $t_0 -s < -s/8$, and that $r \mapsto r^{-\frac{n+1+a}{2}} e^{-1/r}$ is uniformly bounded on $[0, \infty)$.
In case 2), using \eqref{b3} and the first inequality in \eqref{b1}, we obtain
\[
\mathscr G_a(X, X_0,t_0-s) \leq   \frac{C(n, b, A)}{(-s)^{\frac{n+a+1}{2}}(t_0-s)^{\frac{n+1+a}{2}} }e^{-\frac{1}{8(t_0-s)} }  e^{- \frac{|X|^2}{32(t_0-s)}}  \left(\frac{yy_0}{2(t_0-s)} \right)^{-a/2} e^{-\frac{yy_0}{2(t_0-s)}},
\]
and the desired bound (a) follows again by additionally using that
$r \mapsto r^{-a/2} e^{-r}$ is uniformly bounded in the interval $[c(a),\infty)$.

To prove the estimate for (b) we argue similarly to (a), see also the proof of the second part of  \cite[Claim 7.8]{DGPT}.
At this point, since $\tilde V$, $\nabla \tilde V$ have at most polynomial growth at infinity, by letting $R \to \infty$ and using the bounds in case (a) of Claim \ref{claim1}, we deduce  that the  second integral in \eqref{fund} goes to $0$. Also, using the bounds in case (b) of Claim \ref{claim1} and Cauchy-Schwarz inequality, we can assert  that the first integral in \eqref{fund} converges to the  corresponding  integral  in the right-hand side of  \eqref{tc} as $R \to \infty$. Consequently, the sub-mean value estimate claimed in \eqref{tc} holds.

Now  let   $(X_0,t_0) \in  \B_{Ar}^{+} \times (-r^2/4, 0)$. Let also     $s_0 \in (-r^2/2, -r^2/4]$.  Then by  using     the fact that $|y_0|, |x_0| \leq Ar$, $|s_0|, |t_0+s_0| \sim r^2$ and also that  $|s_0|<  |t_0 +s_0|  <4 |s_0|$, we can use the representation as in \eqref{b3}  and by using the asymptotics in \eqref{b1} ( and also by dividing the considerations into two cases as in the proof of Claim \ref{claim1})  we can  assert   that the following estimate holds
\begin{equation}\label{t00}
p(x, x_0, -s_0 ) p^{(a)}(y,y_0,  -s_0)\leq  C_0 \Gb (X, t_0+s_0) e^{\frac{\langle x,x_0\rangle  }{2|t_0+s_0|}} e^{\frac{C_4yy_0}{2|t_0+s_0|}}
\end{equation}
where $C_4$ is universal and depends also on $a$.

Therefore  by using the estimate \eqref{t00} in the submean value inequality \eqref{tc},  we obtain by letting $t_0+s_0=s$, that the following inequality holds,
\[
\tilde V(X_0,t_0) \leq C_5 \int_{\Rnp} \tilde V(X, s) \Gb (X, s) e^{\frac{\langle x,x_0\rangle  }{2|s|}}  e^{\frac{C_4 yy_0}{2|s|}}y^a.
\]
Then  by using  Cauchy-Schwartz, we obtain
\begin{align}\label{de10}
\tilde V(X_0,t_0) &\leq C_5 \left(\int_{\Rnp}    \tilde V^2(x,y, s) \Gb(X,s) y^a\right)^{1/2} \left(\int_{\Rnp}  \Gb e^{\frac{\langle x,x_0\rangle  }{|s|}} e^{\frac{C_4 yy_0}{|s|}}y^a\right)^{1/2}
\\
& \leq C_7 \left(\int_{\Rnp}   \tilde V^2(X, s)   \Gb (X, s) y^a \right)^{1/2} e^{\frac{C_8|X_0|^2}{|s|}}.
\notag
\end{align}
The last inequality above follows  by  multiplying and dividing the following  integral  \[
\left(\int_{\Rnp}  \Gb e^{\frac{\langle x,x_0\rangle  }{|s|}} e^{\frac{C_4 yy_0}{|s|}} y^a\right)^{1/2}
\]
with $e^{-\frac{|x_0|^2}{s}} e^{-\frac{C_4^2 |y_0|^2}{s}}$ and completing squares in the exponent, then by using Fubini and  change of variables.

Now by using $|X_0| \leq Ar$ and  $|s| \sim r^2$, we can deduce from \eqref{de10} that the following estimate holds  for some universal $C_{10}$,
\[
\tilde V(X_0,t_0) \leq C_{10} \left(\int_{\Rnp}  \tilde V^2(x,y, s)    \Gb (X, s) y^a  \right)^{1/2}.
\]
Subsequently, by integrating from $s \in (-r^2, -r^2/2)$ and applying the Cauchy-Schwartz inequality as in the proof of Lemma~9.3 in \cite{DGPT}, we obtain
\begin{align}\label{cmp2}
\tilde V(X_0,t_0) &\leq C \left(\frac{1}{r^2} \int_{-r^2}^{-r^2/2} \int_{\Rnp}  \tilde V^2 \Gb y^a\right)^{1/2}
\\
& \leq C H(\tilde V, r)^{1/2} \leq  C H(U, r)^{1/2} + Cr^{\ell}.
\notag
\end{align}
Since $(X_0, t_0) \in \B_{Ar}^{+} \times (-r^2/4, 0)$ is arbitrary, the conclusion of the lemma now  follows from \eqref{cmp2} and the expression  of $\tilde V$ in terms of $V$.
\end{proof}

We also need the following two lemmas in our blowup analysis in the proof of Theorem~\ref{blowup}.

\begin{lemma}\label{non}
With $\kappa$ as in Lemma~\ref{L:2plus2sigma}, let $\kappa'$ be such that $\kappa < \kappa' < \ell -1 +\sigma$  and $H(U,r) \geq r^{2\ell -2 +2\sigma}$. Then
there exists $r_U>0$ depending  on $\kappa', \sigma$ such that
\begin{equation}\label{ast}
\begin{cases}
H(U_r, \rho) \geq \rho^{2\kappa'}&\text{for $0 < \rho \leq 1$, $0< r < r_U$}
\\
H(U_r,R) \leq R^{2\kappa'}&\text{for any $R\geq 1$, $0 < r< \frac{r_U}{R}$}
\end{cases}
\end{equation}

\end{lemma}

\begin{proof}
Since
\[
2\kappa = \lim_{r\to0} \frac{rH'(U,r)}{H(U,r)},
\]
there exists $r_0>0$ depending on $U, \kappa'$ such that for $0<r<r_U<r_0$ we have
\[
\frac{H'(U,r)}{H(U,r)} \leq 2\frac{\kappa'}{r}.
\]
Then, by integrating from $\rho r$ to $r$ and subsequently by exponentiating the corresponding inequality  we obtain
\[
\frac{H(U,r)}{H(U,\rho r)} \leq \rho^{-2\kappa'},
\]
which implies
\[
H(U_r, \rho)= \frac{H(U,r\rho)}{H(U, r)} \geq \rho^{2\kappa'}.
\]
The second estimate  in \eqref{ast} follows similarly by integrating from $r$ to $rR$ and by noting that $rR < r_U$.
\end{proof}

\begin{lemma}\label{F1}
Under the hypothesis of Lemma~\ref{non} above,  with $F_r$ as in \eqref{Fr}, we have that for any $R\geq 1$ such that $0 < Rr\leq r_0$, the following estimate holds for some universal constant $C$ and where $\kappa'$ is as in Lemma~\ref{non}.
\[
\int_{\Sa_R^+} t^2  F_r^2 \Gb y^a \leq C R^{4+ 2\kappa'} r^{2- 2\sigma}
\]

\end{lemma}

\begin{proof}
We have that
\begin{align}
\int_{\Sa_R^+} t^2  F_r^2 \Gb y^a&= \frac{r^4}{H(U,r)} \int_{\Sa_R^{+}}  t^2 F(rX, r^2 t)^2 \Gb y^a
\\
& =\frac{1}{r^2 H(U,r)}  \int_{\Sa_{Rr}^+}  t^2 F^2 \Gb y^a\leq  \frac{1}{r^2H(U, Rr)}  (Rr)^{2+ 2\ell} \frac{H(U, Rr)}{H(U,r)}
\notag
\\
& \leq C_F R^{4+2\kappa'}  r^{2-2\sigma},
\notag
\end{align}
where we used the fact since  $Rr \leq r_U < r_0$, implying that $H(U, Rr) \geq (Rr)^{2\ell -2+ 2\sigma}$.
\end{proof}

From this point on, we need to assume $\ell\geq 4$ as well as that $F$ satisfies the bounds \eqref{aF}, \eqref{aF-2}, \eqref{aF-3}, unless stated otherwise.
We then have the following theorem concerning the existence and homogeneity of blowups.

\begin{thrm}[Existence and homogeneity of blowups]\label{blowup}
Let $U \in \Sp$ with $F$ satisfying \eqref{aF}, \eqref{aF-2}, \eqref{aF-3} for some $\ell \geq 4$.  Fix $\sigma \in (0, 1)$. Suppose
\[
 \kappa= \Phi_{\ell, \sigma}(U, 0^+) <  \ell -1 + \sigma.
\]
For $r>0$, let $U_r$ denote the Almgren rescalings as in Definition \ref{D:Ahomrescalings}. Then
\begin{itemize}
\item[(i)] For every $R>0$, there exists $r_{R, U}>0$ such that
\begin{align*}
\int_{\Sa_{R}^+}   ( U_r^2 + |t| |\nabla U_r|^2 + |t|^2 ( |\nabla (U_r)_{x_i}|^2 + |t|^2 (U_{r})_t^2) \Gb y^a \leq C(R),&\quad 0< r< r_{R, U}\\
\intertext{and}
\int_{\Q_{R}^+}  ( U_r^2 + |t| |\nabla U_r|^2 + |t|^2 ( |\nabla (U_r)_{x_i}|^2 + |t|^2 (U_{r})_t^2)  y^a \leq C(R),&\quad 0< r< r_{R, U};
\end{align*}
\item[(ii)] There exists a sequence $r_j \to 0$  and a function $U_0 \in \Sa_{\infty}^{+}$ such that
\[
\int_{\Sa_{R}^+}( (U_{r_j}- U_0)^2 + |t| |\nabla U_{r_j} - \nabla U_0|^2) \Gb y^a \to 0;
\]
\item[(iii)] $U_0$ is  parabolically homogeneous of degree $\kappa$ and  is a global solution of  the homogeneous thin obstacle problem, i.e.,
\begin{equation}\label{U0gb}
\begin{cases}
\La U_0 =0 &\text{in $\Sa_{\infty}^+$}
\\
\min\{U_0, -\p_y^a U_0\}=0 &\text{on }S_\infty.
\end{cases}
\end{equation}

\item[(iv)] $U_0, \nabla_x U_0, \p^a_y U_0$ are  continuously defined up to $\{t=0\}$ and $\partial_t U_0$ is bounded up to $\{t=0\}$.
\end{itemize}
\end{thrm}

\begin{proof}
We first note that $U_r$ solves \eqref{g1}. Therefore by taking $r_{R,U}=\frac{r_U}{R}$, the first estimate in   (i) in Theorem~\ref{blowup} follows from  the Gaussian estimates in  Lemma~\ref{Est1}, the second  estimate in  Lemma~\ref{non}, and  Lemma~\ref{F1}.

To show the second estimate in (i), we observe that arguing as Lemma~\ref{Est1} one can also establish the following ``unweighted'' version of second derivative estimates
\begin{equation}\label{s1}
\int_{\Q_R^+} (  |\nabla U_r|^2  + |\nabla (U_r)_{x_i}|^2 + (U_r)_t^2 )  |y|^a \leq C(n,a, R) \int_{\Q_{2R}^+} (U_r^2  +  F_r^2)  |y|^a,
\end{equation}
for any $R>0$.
We next note that Lemma~\ref{hr}  coupled with the fact that
\begin{equation}\label{eq:HU-est-below}
H(U, r) \geq  r^{2\ell- 2 + 2\sigma}
\end{equation}
for small enough $r$ implies that $U_r$ is bounded up to $\{t=0\}$.  Moreover, since $F$ satisfies \eqref{aF}, then again by \eqref{eq:HU-est-below} we can deduce  that $F_r$ are  uniformly bounded for small enough $r$.  Consequently we can assert that both $U_r$ and $F_r$ are uniformly bounded in $L^{2}( \Q_{2R}^{+}, y^a dXdt)$ independent of $r$ for any $R>0$, provided $r \leq r_0$  for some $r_0$ sufficiently small. This implies the second estimate in (i).

In view of Lemma~\ref{Est2} and Lemma~\ref{F1},  in order to establish (ii),  it suffices to show the existence of $U_0$ and the convergence
\begin{equation}\label{ai}
\int_{\Sa_{R}^+}|U_{r_j} - U_0|^2\Gb y^a \to 0
\end{equation}
for a subsequence $r_j \to 0$.
Since $\kappa\geq 0$, for all small enough $r$, say $r \leq r_1$, we have
\[
\frac{rH'(U, r)}{H(U, r)} \geq -1.
\]
Integrating the above inequality from $r \delta$ to $r$ we obtain
\[
H(U_r, \delta)= \frac{H(U,r\delta)}{H(U, r)} \leq \delta^{-1}
\]
and consequently
\[
\int_{\Sa_{\delta}^+}  U_r^2 \Gb y^a \leq \delta,\quad  0 < r < r_1.
\]
At this point, we need the following inequalities from \cite{DGPT} which are  corollaries of L. Gross' log-Sobolev inequality (see Lemma~7.7 in \cite{DGPT}). We first  write
\begin{equation}\label{prod}
\Gb(X,t) =\frac{(4 \pi)^{1/2}}{2^a \Gamma(\frac{a+1}{2})}t^{-a/2}p_n(x,t) p_1(y,t),
\end{equation}
where $p_n(x,t) = (4\pi |t|)^{-n/2} e^{\frac{|x|^2}{4t}}$ and $p_1(y,t) = (4\pi |t|)^{-1/2} e^{\frac{y^2}{4t}}$ respectively indicate the backward heat kernels in $\Rn$ and in $\R$.

As in \cite{BG}, we also let
\[
\Ga(X,t)=p_{n}(x,t) p_1(y,t).
\]
The following inequalities hold:
\begin{multline}\label{ineq}
\log\left(\frac{1}{\int_{|f| > 0} \Ga(\cdot,s)} \right) \int_{\Rnn}  f^2 \Ga(\cdot,s) \leq 2 |s| \int_{\Rnn} |\nabla f|^2 \Ga(\cdot,s),\\
\text{for}\ f \in W^{1,2}(\Rnn, \Ga(\cdot,s)),
\end{multline}
and
\begin{multline}\label{ineq1}
\log\left(\frac{1}{\int_{|f| > 0} p_n(\cdot,s)} \right) \int_{\Rn}  f^2 p_n(\cdot,s) \leq 2 |s| \int_{\Rn} |\nabla f|^2 p_n(\cdot,s),\\
\text{for}\  f \in W^{1,2}(\Rn,G(\cdot,s)).
\end{multline}
We now choose  $A> 2$  large  enough such that for all $-1 < t < 0$,
\begin{equation}\label{cond}
 \int_{\R^{n+1} \setminus \B_{A/2}}  \Ga (X, t) dX \leq e^{-1/\delta},\quad \int_{\R^{n} \setminus B_{A/2}} p_n(x,t) dx \leq e^{-1/\delta}.
\end{equation}
Using the uniform gradient estimates from (i), i.e.,
\[
\int_{\Sa_{R}^+} |t| |\nabla U_r|^2 \Gb y^a < C_R
\]
and the  inequalities  \eqref{ineq} and \eqref{ineq1},  we can argue as in   (7.17)--(7.26) in \cite{BG}, which crucially uses the estimate  \eqref{cond}, to conclude that the following holds,
\begin{equation}\label{conv10}
\int_{[(\Rnp \setminus \B_A) \times (-R^2,0)]  \cup \Sa_{\delta}^{+}}   U_r^2(X,t) \Gb(X,t) y^a\leq C \delta
\end{equation}
for some universal $C$, which also depends on $R$. Now in the set  $E=\B_A^{+} \times [-R^2, -\delta^2]$, which is the complement of $[(\Rnp \setminus \B_A) \times (-R^2, 0)]  \cup \Sa_{\delta}^{+}$, we have that $\Gb$ is bounded from above and below. Therefore from the uniform Gaussian estimates as in (i), we have that $\{U_r\}$ is  uniformly bounded in $W^{2,1}_2(E, y^a dXdt)$. As a consequence, we can extract a subsequence which converges strongly to some $U_0$ in $L^{2}(E, y^a dXdt)$  and consequently in $L^{2}(E, y^a \Gb dXdt)$. Hence the claim in (ii) now follows in a standard way by a Cantor diagonalization argument by letting $\delta \to 0$ and $A \to \infty$.

\medskip
We now prove the claim in (iii).  Given any compact subset  $K$ of $\overline{\Sa_\infty}=\overline{\Rnp} \times (-\infty, 0]$,  the second estimate in i) yields that  $\{U_r\}$ is uniformly bounded in $W^{2,1}_2(K, y^a dX dt)$. Then   we can apply the local regularity  estimates  in \cite{BDGP2} to assert that  for some $\gamma=\gamma(a,n)$, $\nabla_x U_r,  y^a \partial_y U_r \in H^{\gamma, \gamma/2}(K), \partial_t U_r \in L^{\infty}(K)$ uniformly in $r$. This follows from the fact that the conditions \eqref{aF}, \eqref{aF-2}, \eqref{aF-3} imply that $F_r$, $\nabla_X F_r$, and $\p_t F_r$ are locally uniformly bounded in $\overline{\Sa_\infty}$. Thus, for $\p_tF_r$ we have
\begin{align*}
|\p_tF_r(X,t)|=\frac{r^4 |\p_t F(rX,r^2t)|}{H(U,r)^{1/2}}\leq\frac{C_\ell r^\ell |(X,t)|^{\ell-4}}{r^{\kappa'}}\leq C_\ell r^{1-\sigma}|(X,t)|^{\ell-4},
\end{align*}
for small $r>0$, which gives the uniform bound on compact subsets of $\overline{\Sa_\infty}$.
The above uniform regularity of $U_r$ is enough to pass to the limit in the Signorini problem and infer that $U_0$ solves the Signorini problem in \eqref{U0gb}.  Moreover, by lower semicontinuity,  the following estimate for $U_0$ holds
\begin{equation}\label{u0}
\int_{\Sa_{R}^+}   ( U_0^2 + |t| |\nabla U_0|^2 + |t|^2 ( |\nabla (U_0)_{x_i}|^2 + |t|^2 (U_{0})_t^2) \Gb y^a \leq C(R),
\end{equation}
for any  $R>0$.  With \eqref{u0} at our disposal, we can justify  the Poon type computations for $U_0$ by using truncations as in the appendix of \cite{BG}. Here, we note that the intermediate calculations for the corresponding truncated functionals  can be justified  using  the fact that  $\nabla_x U_0,  y^a \partial_y U_0\in H^{\gamma, \gamma/2}(K), \partial_t U_0 \in L^{\infty}(K)$ for any compact subset $K$ of $\overline{\Rnp} \times (-\infty, 0]$. Therefore, we can  infer
\begin{equation}\label{poon}
N'(U_0, r)= \frac{1}{r^3 H(U_0, r)} \left(  \int_{\Sa_r^+}  U_0^2 \Gb y^a \int_{\Sa_r^+}  (ZU_0)^2   \Gb y^a - \left(\int_{\Sa_r^+}  U_0 ZU_0 \Gb y^a\right)^2   \right)
\end{equation}
for any $r>0$. Keeping (ii) in mind, we  conclude
\begin{equation}\label{cov10}
\left\{\begin{aligned}
I(U_r, \rho) &\to I(U_0, \rho)
\\
H(U_r, \rho) &\to H(U_0, \rho).
\end{aligned}
\right.
\end{equation}
Since
\[
H(U_r, \rho) \geq \rho^{2\kappa'}
\]
for any $\kappa' > \kappa$  by Lemma~\ref{non}, we have $H(U_0, \rho) \neq 0$ for any $\rho>0$.  Now we can infer from \eqref{cov10} that
\[
N(U_0, \rho) = \lim_{r\to 0^+} N(U_r, \rho) = \lim_{r\to 0^+} N(U, r \rho)= \kappa,
\]
i.e.,
\begin{equation}\label{hm2}.
N(U_0, \cdot) \equiv  \kappa
\end{equation}
It follows that $N'(U_0,\cdot) \equiv 0$ and hence the right hand side in \eqref{poon}  vanishes. In turn, from the equality in the Cauchy-Schwartz inequality, we have that
\[
\int_{\Sa_r^+}  U_0^2 \Gb y^a \int_{\Sa_r^+}  (ZU_0)^2   \Gb y^a = \left(\int_{\Sa_r^+} U_0 ZU_0 \Gb y^a\right)^2
\]
implies
\begin{equation}\label{hm1}
ZU_0= \kappa_0 U_0.
\end{equation}
From the representation
\[
I(U_0, r) = \frac{1}{2r^2} \int_{\Sa_r^+} U_0 ZU_0 \Gb y^a,
\]
coupled with \eqref{hm2} and \eqref{hm1}, we obtain  that
\[
\kappa_0=\kappa
\]
and therefore $U_0$ is parabolically homogeneous of degree $\kappa$.   This finishes the proof of (iii).

\medskip
To conclude, we note that (iv) follows from the   second estimate in (i), which is uniform in $r$. Hence, the local regularity estimates  developed in \cite{BDGP2} imply that,  for any compact subset  $K$ of $\overline{\Rnp} \times (-\infty, 0]$,  $\nabla_x U_r, y^a \partial_y U_r \in H^{\gamma, \gamma/2}(K)$ (for some $\gamma=\gamma(a,n)$), $\partial_t U_r \in L^{\infty}(K)$ uniformly in $r$.  By Ascoli-Arzel\`a's theorem,  we have that  for a  subsequence of $\{r_j\}$ as in (ii),  $U_{r_j}$,   $\nabla_x U_{r_j}$ and $y^a \partial_y U_{r_j}$ converge uniformly in $K$   to $U_0$, $\nabla_x U_0$ and $y^a \partial_y U_0$  respectively. Thus, (iv) holds.  Note also that $\partial_t U_0 \in L^{\infty}(K)$ follows from the uniform convergence  of $U_{r_j}$ to $U_0$ and  the uniform time   Lipschitz bounds  for $U_{r_j}$'s.
\end{proof}

\section{Homogeneous global solutions and regular points}\label{S:global}

In this section, we show that the frequency limit at a free boundary point is either $\kappa = 1+s$ or $\kappa \geq 2$.  Furthermore, we show that the free boundary is regular near points where $\kappa=1+s$.
\begin{thrm}\label{min}
Let $\ell \geq 4$ and $\sigma >0$. Then with $\kappa$ as in Theorem~\ref{blowup}, we have that
\[
\kappa \geq 1+s.
\]
\end{thrm}

\begin{proof}
  Since $(0,0) \in \Gamma_{*}(U)$, we have
  \begin{equation}\label{zero}
  U(0,0,0)= \nabla_x U(0,0,0)= \p_y^a U(0,0,0)=0
  \end{equation}

  Using the boundedness of $U_t$,  we have that $U(\cdot, 0)$ solves the elliptic thin obstacle problem with bounded right hand side. Consequently, from the regularity results in the elliptic case in \cite{CDS} and \eqref{zero} we infer that for some universal $C$
\[
|U(X,t)| \leq C(|X|^{1+s}+|t|) \leq C(|X|^2+|t|)^{(1+s)/2}.
\]
Hence,
\begin{align}\label{e101}
H(U,r) &\leq \frac{C}{r^2} \int_{\Sa_r^+}  (|X|^2 + |t|)^{1+s} \Gb y^a
\\
& \leq Cr^{2+2s},
\notag
\end{align}
where the second inequality in \eqref{e101} follows from  a change of variable and the homogeneity property of $\Gb$.  Now with $\kappa$ as in Theorem~\ref{blowup}, we obtain from the non-degeneracy Lemma~\ref{non} that
\begin{equation}\label{e102}
H(U, r) \geq r^{2\kappa'}
\end{equation}
for any $\kappa'>\kappa$. Then \eqref{e101} and \eqref{e102} together imply
\[
\kappa' \geq 1+s.
\]
The conclusion follows  letting $\kappa' \to \kappa$.
\end{proof}

Next, we  state our gap lemma.

\begin{lemma}\label{gap1}
Let  $\ell \geq 4$ and  $\kappa$ be as in Theorem~\ref{blowup}. Then  either  $\kappa = 1+s$ or $\kappa \geq 2$.
\end{lemma}

\begin{proof}
 Assume on the contrary that $\kappa < 2$. This also implies $\kappa<\ell-1+\sigma$. Then it follows  that for  any $\kappa' $ such that $\kappa < \kappa' <2$, we have that
 \[
 H(U, r) \geq \left(\frac{r}{r_0}\right)^{2\kappa'}  H(U, r_0),
 \]
 for small enough $0<r<r_0$, i.e.,
 \begin{equation}\label{bl1}
 H(U, r) \geq  c_0 r^{2\kappa'},\quad\text{with}\
c_0= \frac{H(U, r_0)}{r_0^{2 \kappa'}}.
 \end{equation}
 As before, note that $U_r$ solves the Signorini problem corresponding to $F_r$ as in \eqref{Fr}. Using \eqref{aF} and \eqref{bl1}, it follows that $F_r  \to 0$ as $r \to 0$. In addition,
 \begin{equation}\label{zr}
 \partial_t U_r = \frac{r^2 U_t ( rX, r^2 t)}{\sqrt{H(U, r)}} \to 0\quad \text{as $r \to 0$},
 \end{equation}
 since $U_t$ is bounded and
 \[
 \frac{r^2}{\sqrt{H(U, r)}} \to 0\quad\text{as }r\to 0,
 \]
 because of \eqref{bl1} with $\kappa'<2$. This implies that $U_0$ is a time independent  global solution of the   Signorini problem   and is  homogeneous of degree $\kappa$, which is less than $2$. Then it follows from the classification result in Theorem~5.7  in \cite{CSS} that $\kappa=1+s$.
 \end{proof}

We now show that  near points with frequency $\kappa=1+s$, the free boundary is $H^{1+\alpha, \frac{1+\alpha}{2}}$ regular for some $\alpha>0$ by invoking the elliptic theory. More precisely, we recall once more than $U(\cdot, 0)$ solves the elliptic Signorini problem with bounded right hand side because of the boundedness of $U_t$. We show that when the parabolic frequency limit $\kappa$  at $(0,0)$ equals $1+s$, also the elliptic Almgren frequency  at $0 \in \Gamma(U(\cdot, 0))$ (say $\kappa_0$)  equals $1+s$. From this,  it follows that  the free boundary is $H^{1+\alpha,(1+\alpha)/2}$ regular near $(0, 0) \in \Rn \times (-\infty, 0]$ in $x, t$ for some $\alpha>0$. We refer to \cite{BDGP2} for a rigorous  justification of $H^{1+\alpha,(1+\alpha)/2}$-regularity   of the free boundary in space and time near  such  an  elliptic  regular point for $U(\cdot, 0)$. We just mention here that this result crucially   uses the elliptic epiperimetric inequality  developed  in \cite{GPPS}, coupled with the boundedness of $U_t$.  The corresponding result can be stated as  follows.

\begin{thrm}\label{regular point}
Let $U \in \Sp$ with $F$ satisfying \eqref{aF}, \eqref{aF-2}, \eqref{aF-3} for some $\ell \geq 4$. Assume that $(0,0) \in \Gamma_{*}(U)$ and let  $\kappa$ be as in Theorem~\ref{blowup}. If $\kappa= 1+s$, then $\Gamma_{*}(U)$  is $H^{1+\alpha,(1+\alpha)/2}$-regular  near $(0,0)$ for some $\alpha>0$.
\end{thrm}
\begin{proof}
Let $\kappa<\kappa'<2$. For $0<\delta<1$, rewrite the integral in the definition of $H(U, r)$ as follows
\begin{equation}\label{ca1}
H(U, r) = \frac{1}{r^2} \int_{\Sa_{r}^+ \cap \{(X,t)\mid|X| \leq |t|^{\delta/2}\}}  U^2 \Gb y^a +  \frac{1}{r^2} \int_{\Sa_{r}^+ \cap \{(X,t)\mid|X| \geq |t|^{\delta/2}\}}  U^2 \Gb y^a
\end{equation}
with the idea of estimating the second integral in the right hand side  of \eqref{ca1}. By writing  $\frac{1}{t^{\frac{n+a+1}{2}}} e^{\frac{|X|^2}{4t}}$ as
\[
\frac{1}{t^{\frac{n+a+1}{2}}} e^{\frac{|X|^2}{4t}} = e^{\frac{|X|^2}{8t}} \frac{1}{t^{\frac{n+a+1}{2}}} e^{\frac{|X|^2}{8t}}
\]
we  note that in the region $\{(X,t)\mid t  < 0, |X| \geq |t|^{\delta/2}, |t| \leq r^2\}$, we have the bound
\begin{equation}\label{ca2}
e^{\frac{|X|^2}{8t}} \leq e^{-\frac{1}{8r^{2-2\delta}}}.
\end{equation}
Therefore,  using \eqref{ca2} and the boundedness of $U$, the second integral on the right hand side of \eqref{ca1} can be estimated as follows
 \begin{equation}\label{ca4}
 \frac{1}{r^2} \int_{\Sa_{r}^+ \cap \{(X,t)\mid|X| \geq |t|^{\delta/2}\}}  U^2 \Gb y^a \leq Ce^{-\frac{1}{8r^{2-2\delta}}}.
 \end{equation}
Consequently, we  obtain
\begin{equation}\label{bll10}
H(U, r) \leq \|U\|^{2}_{L^{\infty}(B_{r^{\delta}}^{+}(0) \times (-r^2, 0))} + C\exp\left({\frac{-1}{8r^{2(1-\delta)}}}\right),
\end{equation}
where $C$ also depends on the global bounds of $U$. Next note that for $0<r< r_1$, where $r_1=r_1(\delta)$ is small enough, we have
\begin{equation}\label{ca5}
C\exp \left({\frac{-1}{8r^{2(1-\delta)}}}\right) \leq \frac{1}{2}  c_0 r^{2\kappa'}
\end{equation}
and therefore we can deduce from \eqref{bl1}, combined with \eqref{bll10} and \eqref{ca5}, that
\[
C r^{2\kappa'} \leq \|U\|^{2}_{L^{\infty}(B_{r^{\delta}}^{+}(0) \times (-r^2, 0))}.
\]
Since $\kappa'<2$ and $U_t$ is bounded, we obtain by letting $r^{\delta}$ as our new $r$ that, for small enough $\tilde r=\tilde r(\delta)$ and all $r \leq \tilde r$, the following inequality holds
\[
\|U(, 0)\|_{L^{\infty}(B_r)} \geq  C r^{\frac{\kappa'}{\delta}}.
\]
Since $\delta$ can be chosen arbitrarily close to $1$ and $\kappa'$  can be chosen arbitrarily close to $\kappa$,  by  letting $\frac{\kappa'}{\delta}$ as our new $\kappa'$, we deduce that there exists $r_2$ small enough, depending also on $\kappa'$, such  that  for $ r \leq r_2$
\begin{equation}\label{bl}
\|U(, 0)\|_{L^{\infty}(B_r)} \geq  C r^{\kappa'}.
\end{equation}
We now  claim that  \eqref{bl} implies that $0 \in \Gamma_{*}(U(\cdot, 0))$ is a regular free boundary point for the corresponding elliptic problem. If not,  then it follows from \cite{CSS} that the  elliptic  Almgren frequency limit $\kappa_0$  for $U(\cdot, 0)$ as in \cite{CSS} or \cite{CDS}  is bigger than or equal to $2$.  This follows from the classification  result  for global time independent  solutions in \cite[Theorem~5.7]{CSS}. Then from the estimate   in \cite[Lemma~6.5]{CSS}, we obtain that $U(\cdot, 0)$ separates from the free boundary at a rate which is at least  quadratic,  and this   is a contradiction to  \eqref{bl} above since $\kappa' < 2$. Therefore  the elliptic frequency limit $\kappa_0$ necessarily equals  $1+s$. The regularity result for the free boundary  in \cite{BDGP2} implies that $\Gamma_*(U)$ is $H^{1+\alpha,(1+\alpha)/2}$-regular near $(0,0)$ in space and time for some $\alpha >0$.
\end{proof}

\begin{rmrk}\label{R:ell-par-hyp}
The proof of Theorem~\ref{regular point} can be viewed as the consolidation of ``parabolic'' and ``elliptic'' approaches to the definition of regular points. Namely, we say that $(x_0,t_0)\in \Gamma_*(U)$ is a \emph{parabolic regular point} if the parabolic frequency $\kappa_U(x_0,t_0)=\Phi_{\ell,\sigma}(U,0^+)=1+s$. We say that $(x_0,t_0)$ is an \emph{elliptic regular point} if $x_0\in \Gamma(U(\cdot,t_0))$ with elliptic frequency $\kappa_{U(\cdot,t_0)}=1+s$. The proof of Theorem~\ref{regular point} shows that these two notions are in fact equivalent. These points can also be defined as free boundary points where the quantities
\begin{align*}
L_{\rm ell}&=\limsup_{r \to 0} \frac{ \|u\|_{L^{\infty}(B_r(x_0)\times\{t_0\})}}{r^{1+s}},\\
L_{\rm par}&=\limsup_{r \to 0} \frac{ \|u\|_{L^{\infty}(Q_r(x_0, t_0))}}{r^{1+s}}
\end{align*}
are bounded away from zero and infinity, i.e.,\ $0<L_{\rm ell}<\infty$, $0<L_{\rm par}<\infty$, for elliptic and parabolic regular points, respectively.

There is also a third approach, taken by the authors of  \cite{ACM}, which we call ``hyperbolic''. Namely, we say that $(x_0, t_0)\in \Gamma_*(U)$ is a \emph{hyperbolic regular point} if the quantity
\begin{equation*}
L_{\rm hyp}=\limsup_{r \to 0} \frac{ \|u\|_{L^{\infty}(B_r^{*}(x_0, t_0))}}{r^{1+s}},
\end{equation*}
with $B_r^{*}(x_0, t_0)= \{ (x,t)\mid (x-x_0)^2 + (t-t_0)^2 \leq r^2\}$,
is bounded away from zero and infinity, i.e.,\ $0 < L_{\rm hyp} < \infty$. It is proved in \cite{ACM} that near such points the free boundary is $C^{1, \alpha}$-regular in space and time. Because of this regularity, it is possible to see that hyperbolic regular points are also elliptic (and equivalently) parabolic regular. The converse statement that elliptic (or parabolic) regular points are hyperbolic is not immediately obvious. However, we should point out that in the case when $s=1/2$ (or equivalently $a=0$), the converse statement does hold because of the higher regularity of the free boundary near (parabolic) regular points, see \cite{BGZ}.
\end{rmrk}

We close this section with a Liouville type result  for the operator $\La$ which will be used subsequently in the classification of singular points.

\begin{lemma}[Liouville type theorem]\label{lv}
Let $v$ be a solution to
\[
\La v=0 \ \text{in $\Sa_\infty=\Rnn \times (-\infty, 0]$}
\]
such that $v(x, y, t) = v(x, -y, t)$ and $|v(X, t)| \leq  C\left(|X|^2 + |t| \right)^{k/2}$. Then $v$ is a polynomial.
\end{lemma}
\begin{proof}
The proof is similar to that of the elliptic case  as in \cite{CSS} and is based on induction in the degree $k$. The following elementary fact will be used:

\medskip
\noindent
\emph{Fact}:  If $\nabla_x v, \nabla_t v$ are polynomials and $v(0, y)$ is a polynomial, then $v$ is a polynomial.

\medskip

Suppose $k \leq 1$. We first note that the  estimate
\begin{equation}\label{y1}
\sup_{Q_{r/2}(0, 0)} ( | D^2_{x} v| + |v_t|) \leq \frac{C}{r^2} \|v\|_{L^{\infty}(Q_r)}
\end{equation}
  follows from the H\"older regularity result of \cite{CSe},  coupled with the  translation invariance of the equation in $(x, t)$.  Now, since  $\La v=0$, \eqref{y1} implies
\begin{equation}\label{y2}
\sup_{Q_{r/2}(0, 0)} \left( v_{yy} + \frac{a}{y} v_y\right) \leq \frac{C}{r^2} \|v\|_{L^{\infty}(Q_r)}.
\end{equation}
Letting $r \to \infty$ (when $k=1$), we obtain from \eqref{y2}
\[
v_{yy} + \frac{a}{y}v_y \equiv 0,
\]
and from \eqref{y1} that $v$ is
time independent. By repeating the arguments in the proof of
Lemma~2.7 in \cite{CSS}, which only  uses the symmetry of $v$ in $y$,
we can assert that $v= bx + c$ in this case. Now for general $k$
(assuming the assertion of the lemma holds up to $k-1$), it follows from the following rescaled estimate\begin{equation}
\sup_{Q_{r/2}(0, 0)} ( r |\nabla_x v| + |v_t|) \leq \frac{C}{r^2} \|v\|_{L^{\infty}(Q_r)},
\end{equation}
the induction hypothesis, and  the fact that $\nabla_x v$ and $v_t$ solve the same equation,   that $\nabla_x v$ is a polynomial of order $k-1$ and $v_t$ is a polynomial of order $k-2$. Also from \eqref{y2} we obtain
 \[
 \left|v_{yy} + \frac{a}{y} v_y\right| \leq C( |X|^2 + |t|)^{\frac{k}{2}-1}.
 \]
 Now, similarly to the elliptic case, we infer  from the observation
 \begin{equation}\label{y4}
 v_{yy} + \frac{a}{y} v_y= |y|^{-a} \partial_y ( |y|^a v_y)
 \end{equation}
 that $v_{yy} + \frac{a}{y} v_y$ solves the same equation as $v$. Observe here that $w=|y|^a v_y$ solves the conjugate equation $\mathcal{L}_{-a} w=0$ (see for instance \cite{BG}), and therefore $v_{yy} + \frac{a}{y}v_y$, being the twice conjugate of $v$, solves the same PDE as $v$. From the estimate in  \eqref{y2}, the fact that $w= v_{yy} + \frac{a}{y}v_y$ solves $\La w=0$, and the induction hypothesis, we deduce that $v_{yy} + \frac{a}{y} v_y$ is a polynomial of order at most $k-2$. In particular, for $(x,t)=(0,0)$,  $v_{yy} (0,y,0) + \frac{a}{y} v_y(0, y, 0)$ is an even polynomial $p(y)= a_0 + a_2 y^2 + \cdots  +a_{2d} y^{2d}$. Using the expression for $v_{yy} + \frac{a}{y} v_y$  in \eqref{y4} and integrating twice, we obtain
 \[
 v(0,y)= c+ by |y|^{-a} +   \frac{a_0}{2(1+a)} y^2 + \frac{a_2}{2(3+a)} y^4 + \cdots+\frac{a_{2d}}{(2d+2)(2d+1+a)}y^{2d+2}.
 \]
 Next, keeping in mind the evenness of $v$, we infer that $b=0$ and hence $v(0, y)$ is a polynomial. Finally, since $\nabla_x v, v_t$ are polynomials, we conclude that $v$ is a polynomial.
\end{proof}

\section{Classification of free boundary points}\label{S:classification}

Let  the obstacle $\psi$ be of class $H^{\ell, \ell/2}$  and let $V_k$ be as in \eqref{global}. Now given $\sigma < 1$, by repeating the arguments in the proof of Lemma~10.1 in \cite{DGPT} we can show that the limit
\begin{equation}\label{ca7}
\kappa= \Phi_{\ell, \sigma} (V_k, 0^+)
\end{equation}
is independent of the cut-off $\tau$ in the definition of $V_k$.  Therefore, if we denote $\kappa$ in \eqref{ca7} by $\kappa_{U}^{\ell, \sigma}$ (since  this  quantity is independent of the cut-off $\tau$ and consequently independent of $U_k$), we have the  following consistency  result for the truncated frequencies whose proof is exactly the same as in the case   $a=0$  in \cite[Proposition 10.3]{DGPT}:
\begin{equation}\label{ca8}
\kappa_{U}^{\ell, \sigma}(0, 0)=  \min\{\kappa_{U}^{\tilde \ell, \tilde \sigma}(0,0), \ell -1+ \sigma\}
\end{equation}
whenever  $\ell \leq \tilde \ell$ and $\ell - 1 + \sigma \leq \tilde \ell - 1+ \tilde \sigma$.  It follows from  \eqref{ca8} that  if $\psi \in H^{\tilde \ell, \frac{\tilde \ell }{2}}$, then
\begin{equation}\label{tr10}
\sup_{\{(\ell, \sigma) \mid \ell -1 + \sigma < \tilde \ell\}} \kappa_{U}^{\ell, \sigma} (0, 0)
\end{equation}
is well defined and we can define the quantity in \eqref{tr10}  above  to be $\kappa_{U}^{(\tilde \ell)}(0, 0)$. In an analogous way, $\kappa_{U}^{(\tilde \ell)}(x, t)$ can be defined for any $(x,t) \in \Gamma_{*}(U)$. Then,  if $\psi \in H^{\ell, \ell/2}$,  for $\kappa \in [1+s, \ell)$ we define
\begin{equation}\label{cl1}
\Gamma_{\kappa}^{(\ell)} (U)= \{ (x,t) \in \Gamma_{*} (U) \mid \kappa_{U}^{(\ell)} (x,t)= \kappa\}
\end{equation}
In view of \eqref{ca8}, we have the following result on the classification of free boundary points analogous to  in \cite[Proposition 10.7]{DGPT}.

\begin{prop}\label{consistency}
If $\psi \in H^{\tilde \ell, \frac{\tilde \ell}{2}}$, with $\tilde \ell \geq \ell \geq 4$, then
\begin{equation}
\begin{cases}
\Gamma_{\kappa}^{(\ell)}(U)= \Gamma_{\kappa}^{(\tilde \ell)}(U), \ \text{if $\kappa < \ell$}
\\
\Gamma_{\ell}^{(\ell)}(U)= \bigcup\limits_{\ell \leq \kappa \leq \tilde \ell}  \Gamma_{\kappa}^{(\tilde \ell)} (U)
\end{cases}
\end{equation}
\end{prop}
Similarly,  we also have the following characterization of points which are on the extended free boundary $\Gamma_{*}(U)$  but not on the free boundary $\Gamma(U)$.

\begin{prop}\label{extended}
If $\psi \in H^{\ell, \ell/2}$ for $\ell \geq 4$. Let $(x_0, t_0) \in \Gamma_{*}(U)$. Then either $\kappa^{(\ell)}_{U}(x_0, t_0) = 1+s$, or $2 \leq \kappa^{(\ell)}_{U}(x_0, t_0) \leq \ell$. Moreover we have that
\[
\Gamma_{*}(U) \setminus  \Gamma(U) \subset \Gamma^{(\ell)}_{\ell}(U) \cup \bigcup \limits_{m \in \mathbb{N}} \Gamma^{(\ell)}_{2m+1-a} (U).
\]
\end{prop}
\begin{proof}
The first part is nothing but Lemma~\ref{gap1}.

\medskip

Suppose now $(x_0, t_0) \in \Gamma_{*}(U) \setminus \Gamma(U)$ and that the frequency limit $\kappa_{U}^{(\ell)} < \ell$. By translation we may assume that $(x_0, t_0)=(0,0)$.  Then there exists a small $\delta >0$ such that  $U= \psi$ on $Q_{\delta}(0,0)$.   Let $V=V_k$ be as in \eqref{global}. Since $\kappa < \ell$, by Theorem~\ref{blowup} there exists a blow up $U_0$ of $V$ over a sequence $r_j \to 0$.  From the fact that $V=0$ on $Q_{\delta}$  we obtain that $U_0$ vanishes on $\{y=0\}$.  Then we  have that the odd extension $\tilde U_0$ is $\La$ caloric and homogeneous  of degree $\kappa$.  As a consequence, $|y|^a (\tilde U_0)_y$ is $\mathscr{L}_{-a}$caloric, symmetric  and homogeneous of degree $\kappa-1 +a$. From the Liouville theorem Lemma~\ref{lv}, it follows that  $\kappa-1+a$ is an integer and moreover, since $U_0$ satisfies the Signorini condition, we have that $-\p_y^a U_0$   is a non-negative polynomial on $\{y=0\}$. Therefore there are two possibilities, either $\p_y^a U_0$ is identically zero on $\{y=0\}$ or $\kappa-1+a$ is an even integer. The former is not possible because $U_0$ and $\p_y^a U_0$ vanishing identically on $\{y=0\}$  would imply $U_0 \equiv 0$,  because of the strong  unique continuation property. This follows from the proof of Lemma~7.7 in \cite{BG}. Hence, we have $\kappa-1+a$ is even and consequently $\kappa$ is of the form $2m+1-a$ for some $m\in\mathbb{N}$. This finishes the proof of the proposition.
\end{proof}

\section{Singular points}\label{S:singular}
In this section we define the singular free boundary points as the
points of zero Lebesgue density of the coincidence set $\Lambda (U)$.

\begin{dfn}[Singular points]\label{def:parab-sing-points} Let $U\in
  \Sp$ with $F$ satisfying \eqref{aF}, \eqref{aF-2}, \eqref{aF-3} for $\ell\geq 4$. We say that
  $(x_0,t_0)\in \Gamma_*(U)$ is
  \emph{singular} if
$$
\lim_{r\to 0^+}\frac{\mathcal{H}^{n+1}(\Lambda(U)\cap
  Q_r(x_0,t_0))}{\mathcal{H}^{n+1}(Q_r)}=0.
$$
We will denote the set of singular points by $\Sigma(U)$ and call it
the \emph{singular set}. We can further classify singular points
according to the homogeneity of their blowup, by defining
$$
\Sigma_\kappa(U)\overset{\rm def}{=}\Sigma(U)\cap \Gamma_\kappa^{(\ell)}(U),\quad
\kappa<\ell-1+\sigma.
$$
\end{dfn}

The following proposition gives a complete characterization of the
singular points in terms of the blowups and the generalized
frequency. In particular, it establishes that
$$
\Sigma_\kappa(U)=\Gamma_\kappa^{(\ell)}(U)\quad\text{for }\kappa=2m<\ell-1+\sigma,\
m\in\N.
$$

\begin{prop}[Characterization of singular points]
  \label{prop:char-sing-point}
Let $u\in\Sp$ with $F$ satisfying \eqref{aF}, \eqref{aF-2}, \eqref{aF-3} for some $\ell\geq 4$
and $0\in \Gamma^{(\ell)}_\kappa(u)$, with
  $\kappa<\ell-1+\sigma$ for some $\sigma\in(0,1)$. Then, the following statements are equivalent:
  \begin{itemize}
  \item[(i)] $0\in \Sigma_\kappa(U)$.
  \item[(ii)] any blowup of $U$ at the origin is a nonzero parabolically
    $\kappa$-homogeneous polynomial $p_\kappa$ in $\Sa_\infty$
    satisfying
$$
\La p_\kappa=0,\quad p_\kappa(x,0,t)\geq
0,\quad p_\kappa(x,-y, t)=p_\kappa(x,y,t).
$$
(We denote this class by $\Pk$, see Definition \ref{D:poly}.)
\item[(iii)] $\kappa=2m$, $m\in\N$.
\end{itemize}
\end{prop}

\begin{proof} (i) $\Rightarrow$ (ii) Note that the rescalings $U_r$
  satisfy
  $$
  \La U_r=|y|^aF_r-2(\partial^a_yU)
  \mathcal{H}^{n+1}\big|_{\Lambda(U_r)}\quad\text{in }\Sa_{1/r},
  $$
  in the sense of distributions, after an even reflection in the $y$
  variable. Since $U_r$ are uniformly bounded in
  $W^{2,1}_2(\Q_{2R}^+,|y|^a dXdt)$
  for small $r$ by Theorem~\ref{blowup},
  $\partial_y^a U_r$ are uniformly bounded in $L_2(Q_R)$. On the
  other hand, if $0\in \Sigma(U)$, then
$$
\frac{\mathcal{H}^{n+1}(\Lambda(U_r)\cap
  \Q_R)}{R^{n+2}}=\frac{\mathcal{H}^{n+1}(\Lambda(u)\cap Q_{Rr})}{(Rr)^{n+2}}\to
0\quad\text{as }r\to 0,
$$
and therefore
$$
(\partial^a_y U_r) \mathcal{H}^{n+1}\big|_{\Lambda(U_r)}\to
0\quad\text{in }\Q_R
$$
in the sense of distributions. Further, the bound $|F(x,t)|\leq C_\ell
|(X,t)|^{\ell-2}$ implies that
\begin{align*}
  |F_r(X,t)|&=\frac{r^2|F(rX,r^2t)|}{H_U(r)^{1/2}}\leq \frac{C_\ell
    r^{\ell}}{H_U(r)^{1/2}}|(X,t)|^{\ell-2}\\
  &\leq C r^{\ell-\ell_0} R^{\ell-2}\to 0\quad\text{in }\Q_R,
\end{align*}
where $\ell_0=\ell-(1-\sigma)/2\in (\kappa,\ell)$ and we have used the fact that
$H_U(r)\geq r^{2\ell_0}$ for $0<r<r_U$, by Lemma~\ref{non}. Hence, any
blowup $U_0$ is caloric in $\Q_R$ for any $R>0$, meaning that it is caloric in the
entire strip $\Sa_\infty=\R^{n+1}\times(-\infty,0]$.  On the other hand, by
the characterization of blowups in Theorem~\ref{blowup} (iii), $U_0$ is
homogeneous in $\Sa_\infty$ and therefore has a polynomial growth at
infinity. Then, by the Liouville-type Lemma~\ref{lv}, we can conclude that $U_0$
must be a parabolically homogeneous polynomial $p_\kappa$ of a certain
integer degree $\kappa$. Note that $p_\kappa=U_0\not\equiv 0$ by
construction. The properties of $U$ also imply that that
$p_\kappa(x,0,t)\geq 0$ for all $(x,t)\in S_\infty$ and and
$p_\kappa(x,-y,t)=p_\kappa(x,y,t)$ for all $(x,y,t)\in
S_\infty$. In other words, $U_0=p_\kappa\in\Pk$.

\medskip\noindent (ii) $\Rightarrow$ (iii) Let $p_\kappa$ be a blowup
of $U$ at the origin. Since $p_\kappa$ is a polynomial, clearly
$\kappa\in\N$. Assume now, towards the contradiction, that $\kappa$ is odd. Then, the nonnegativity of
$p_\kappa$ on
$\R^{n}\times\{0\}\times\{-1\}$ implies that $p_\kappa$ vanishes
there identically, implying that $p_\kappa\equiv 0$ on
$S_\infty$. Now, using the even symmetry in $y$ and the fact that $\La
p_\kappa=0$, we are going to infer that $p_\kappa\equiv 0$, contrary to the
assumption that $p_\kappa$ is nonzero. From even symmetry in $y$, we represent
$$
p_\kappa(x,y,t)=\sum_{\substack{(\alpha,k,j)\in\mathbb{Z}_+^n\times\mathbb{Z}_+\times\mathbb{Z}_+\\|\alpha|+2k+2j=\kappa}} c_{\alpha,k,j} x^\alpha y^{2k}t^j,
$$
Now, for $(\alpha,k,j)$ such that $|\alpha|+2k+2j=\kappa$, consider the partial
derivative
$\partial_x^\alpha\partial_t^j p_\kappa$.
Since $\partial_{x_i}$ and $\partial_t$ are derivatives in
directions tangential to the thin space, we conclude that
$$
\La(\partial_x^\alpha\partial_t^j p_\kappa)=0\quad\text{in
}\Sa_\infty,\quad \partial_x^\alpha\partial_t^j
p_\kappa=0\quad\text{on }S_\infty.
$$
We now prove by induction in $k$, that $c_{\alpha,k,j}=0$ for
$k=0,1,\ldots,\lfloor \kappa/2\rfloor$. When $k=0$, we have $|\alpha|+2j=\kappa$ and therefore
$$
\partial_x^\alpha\partial_t^j p_\kappa\equiv \alpha! j!c_{\alpha,0,j}
$$
and from the vanishing of $\partial_x^\alpha\partial_t^j p_\kappa$ on
$S_\infty$, we conclude that $c_{\alpha,0,j}=0$.
Suppose now we know that $c_{\alpha,k',j}=0$ for $0\leq k'<k\leq \lfloor \kappa/2\rfloor$ and show
that it holds also for $k$. Indeed, one consequence from the inductive
assumption is that
$$
\partial_x^\alpha\partial_t^j p_\kappa(x,y,t)=\alpha!j!c_{\alpha,k,j}y^{2k},
$$
which is $a$-caloric if and only if $c_{\alpha,k,j}=0$. Hence, we
can conclude that $p_\kappa\equiv 0$, contrary to our assumption. Thus, we must have $\kappa\in\{2m\mid m\in\N\}$.

\medskip\noindent (iii) $\Rightarrow$ (ii) The proof of this
implication is stated as a separate Liouville-type result in
Lemma~\ref{lem:Monn-homogen-harm} below.

\medskip\noindent (ii) $\Rightarrow$ (i) Suppose that $0$ is not a
singular point and that over some sequence $r=r_j\to 0^+$ we have
$\mathcal{H}^{n+1}(\Lambda (U_r)\cap Q_1)\geq \delta>0$. From the second estimate in (i) in Theorem~\ref{blowup}, the local regularity estimates developed in \cite{BDGP2} and Ascoli-Arzel\`a, by taking a subsequence if
necessary, we may assume that $U_{r_j}$ converges locally uniformly to
a blowup $U_0$. We claim that
$$
\mathcal{H}^{n+1}(\Lambda (U_0)\cap Q_1)\geq \delta>0.
$$
Indeed, otherwise there exists an open set $\mathscr{O}$ in $S_\infty$ with
$\mathcal{H}^{n+1}(\mathscr{O})<\delta$ such that $\Lambda (U_0)\cap
\overline{Q_1} \subset \mathscr{O}$. Then for large $j$ we must have $\Lambda
(U_{r_j})\cap \overline{Q_1} \subset \mathscr{O}$, which is a contradiction,
since $\mathcal{H}^{n+1}(\Lambda (U_{r_j})\cap \overline{Q_1})\geq
\delta > \mathcal{H}^{n+1}(\mathscr{O})$. Since $U_0=p_\kappa$ is a polynomial,
vanishing on a set of positive $\mathcal{H}^{n+1}$-measure on
$S_\infty$, it follows that $U_0$ vanishes identically on
$S_\infty$. But then, repeating the argument at the end of the step
(ii) $\Rightarrow$ (iii), we conclude that $U_0\equiv 0$, a
contradiction. Thus, $0$ is a singular point.

\medskip

The implication (iii) $\Rightarrow$ (ii) in
Proposition~\ref{prop:char-sing-point} is  a consequence of the
Liouville-type result Lemma~\ref{lv} which is the parabolic counterpart of
Lemma~1.3.3 in \cite{GP}.
\end{proof}

This, in turn, is a particular case of the following lemma, analogous
to Lemma~1.3.4 in \cite{GP} in the elliptic case, which stems from
Lemma~7.6 in \cite{Mon2}.

\begin{lemma}\label{lem:Monn-homogen-harm} Let $v\in
  W^{1,1}_{2,loc}(\Sa_\infty,|y|^adXdt)$ be such that $\La v\geq
  0$ in $\Sa_\infty$ and $\La v=0$ in $\Sa_\infty\setminus
  S_\infty$. If $v$ is parabolically $2m$-homogeneous, $m\in\N$, and
  has a polynomial growth at infinity, then $\La v=0$
  in $\Sa_\infty$.
\end{lemma}

\begin{proof} Let $\mu\overset{\rm def}{=}\La v$ in
  $\R^{n+1}\times(-\infty,0)$. By the assumptions, $\mu$ is a nonnegative
  measure, supported on $\{y=0\}\times(-\infty,0)$. We are going to
  show that in fact $\mu=0$. To this end, let $P(x,t)$ be a
  parabolically $2m$-homogeneous $a$-caloric polynomial, which is positive
  on $\{y=0\}\times(-\infty,0)$. For instance, one can take the
  polynomial
$$
p(x,t)=\sum_{j=1}^{n-1} x_j^{2m}+(-t)^m,
$$
and let $P=\widetilde{p}$ be the $a$-caloric extension constructed in
Lemma~\ref{calext}.
Further, let
$\eta\in C^\infty_0((0,\infty))$, with $\eta\geq 0$, and define
$$
\Psi(x,t)=\eta(t) \Gb(X,t).
$$
Note that we have the following identity (similar to that of $\Gb$)
$$
\nabla_X\Psi=\frac{X}{2t}\Psi.
$$
We have
\begin{align*}
  \langle \operatorname{div}(|y|^a\nabla v ),\Psi P\rangle &=-\int_{-\infty}^0\int_{\R^{n+1}}\langle \nabla  v,\nabla(\Psi P)\rangle |y|^a dX\,dt\\
  &=-\int_{-\infty}^0\int_{\R^{n+1}}  [\Psi\langle\nabla v,\nabla P\rangle+P\langle\nabla v,\nabla\Psi\rangle]|y|^a\,dX\,dt\\
  &=\int_{-\infty}^0\int_{\R^{n+1}}(\Psi v \operatorname{div}(|y|^a\nabla P)+|y|^a [v\langle\nabla\Psi,\nabla P\rangle-P\langle\nabla v,\nabla\Psi\rangle])\,dX\,dt\\
  &=\int_{-\infty}^0\int_{\R^{n+1}}\left(v\operatorname{div}(|y|^a\nabla P)+\frac{|y|^a}{2t}
    \left[v\langle X,\nabla P\rangle-P\langle X,\nabla
    v\rangle\right]\right)\Psi\,dX\,dt.
\end{align*}
We now use the identities $\operatorname{div}(|y|^a\nabla P)-|y|^a\partial_t P=0$, $\langle X,\nabla
P\rangle+2t\partial_t P=2m P$, $\langle X,\nabla v\rangle+2t\partial_t v=2m v$ to arrive
at
\begin{align*}
  \langle \operatorname{div}(|y|^a\nabla v),\Psi P\rangle
  &=\int_{-\infty}^0\int_{\R^{n+1}}\left[2m P v-P\langle X,\nabla v\rangle\right]\frac{|y|^a}{2t}\,\Psi\,dX\,dt\\
  &=\int_{-\infty}^0\int_{\R^n}\partial_t v\Psi P|y|^adX\,dt\\
  &=\langle |y|^a\partial_t v,\Psi P\rangle.
\end{align*}
Therefore, $\langle\mu,\Psi P\rangle=\langle|y|^a\partial_t v-\operatorname{div}(|y|^a\nabla v),\Psi
P\rangle=0$. Since $\mu$ is a nonpositive measure, this implies that
actually $\mu=0$ and the proof is complete.
\end{proof}

\section{Weiss and Monneau type monotonicity formulas}\label{S:WM}

In this section we establish two families of monotonicity formulas that play a crucial role in our analysis of singular points. The elliptic ancestors of these formulas were first obtained in \cite{GP} in the study of the Signorini problem corresponding to $a = 0$ (or $s = 1/2$), and were subsequently generalized to all $a\in (-1,1)$ (all $s\in (0,1)$) in \cite{GRO}. In the parabolic setting and still for the case $a = 0$ such formulas were first proved in \cite{DGPT}. Theorems \ref{T:weisspar} and \ref{T:monneau} below respectively extend to all values $a\in (-1,1)$ Theorems 13.1 and 13.4 in \cite{DGPT}.

In the following statement the quantities $H(U,r)$ and $D(U,r)$ are those defined in \eqref{HU} and \eqref{IU} respectively.

\begin{thrm}[Weiss type monotonicity formula in Gaussian space]\label{T:weisspar}
Let $U \in \Sp$ with $F$ satisfying \eqref{aF} for some $\ell\geq 2$ and a constant $C_\ell$.

For $\kappa\in (0,\ell)$ we define the parabolic $\kappa$-Weiss type functional
\begin{equation}\label{kW}
\mathscr W_\kappa(U,r) \overset{\rm def}{=} r^{-2\kappa} \big\{D(U,r) -  \frac{\kappa}2 H(U,r)\big\}.
\end{equation}
Then, for any $\sigma\in (0,1)$ such that $\kappa\leq \ell-1+\sigma$ there exists $C'>0$ depending on $n, a, \ell, C_\ell$ such that
\begin{equation}\label{kW'}
\mathscr W'_\kappa(U,r) \ge \frac{1}{r^{2\kappa + 3}} \int_{\Sa_r^+} \big(ZU - \kappa U + |t| F\big)^2 \Gb y^a - C' r^{1-2\sigma}.
\end{equation}
In particular, with $C = \frac{C'}{2-2\sigma}$ the function
\[
r\mapsto\mathscr W_\kappa(U,r) + C r^{2-2\sigma},
\]
is monotonically nondecreasing in $(0,1)$, and therefore the limit exists
\[
\mathscr W_\kappa(U,0^+) \overset{\rm def}{=} \underset{r\to 0^+}{\lim} \mathscr W_\kappa(U,r).
\]
\end{thrm}

\begin{proof}
Using Lemmas \ref{L:HU'} and \ref{L:IU'} we find
\begin{align*}
r^{2\kappa + 3} \mathscr W'_\kappa(U,r) & = r^3(D'(U,r) - \frac{\kappa}2  H'(U,r)) - 2 \kappa r^2(D(U,r) - \frac{\kappa}2 H(U,r))
\\
& = \int_{\Sa_r^+} \big(ZU - \kappa U + |t| F\big)^2 \Gb y^a  - \int_{\Sa_r^+} |t|^2 F^2\ \Gb y^a.
\end{align*}
Next, we note that \eqref{aF3} gives
\[
\int_{\Sa_r^+} |t|^2F^2\ \Gb y^a \le C r^{2(1+ \ell)},
\]
for some $C>0$ depending only on $n, a, \ell, C_\ell$.
This gives
\[
\mathscr W'_\kappa(U,r) \ge \frac{1}{r^{2\kappa + 3}} \int_{\Sa_r^+} \big(ZU - \kappa U + |t| F\big)^2 \Gb y^a - C r^{-1+2(\ell - \kappa)}.
\]
If now $1-\ell+\kappa\leq \sigma<1$, we conclude that
\[
\mathscr W'_\kappa(U,r) \ge - C r^{1-2\sigma},
\]
and therefore the function
\[
r\mapsto\mathscr W_\kappa(U,r) + C r^{2-2\sigma},
\]
is monotonically nondecreasing.
\end{proof}

In the sequel we will need the following results.
\begin{lemma}\label{l:optH} Under the assumptions of Theorem~\ref{T:weisspar}, suppose in addition that $0\in\Gamma_\kappa^{(\ell)}(U)$ for $\kappa<\ell-1+\sigma$. Then
$$
H(U,r)\leq C\left(\|U\|^2_{L^2(\Sa_1^+,|y|^a)}+C_\ell^2\right)r^{2\kappa},
$$
with $C=C(\kappa,\sigma, n)>0$.
\end{lemma}
\begin{proof} We begin by observing that the following alternative holds: either (i) $H(U,r)\leq r^{2\ell-2+2\sigma}$, or (ii)  $H(U,r)> r^{2\ell-2+2\sigma}$. Since the conclusion follows immediately in case (i), we assume that (ii) holds. Let $(r_0,r_1)$ be a maximal interval in the  open set $\{r\in (0,1)\mid H(U,r)> r^{2\ell-2+2\sigma}\}$. For $r\in (r_0,r_1)$, from Theorem~\ref{T:poon} we infer
\begin{align*}
\Phi_{\ell,\sigma}(U,r) {=} \frac{1}{2} r e^{C r^{1-\sigma}} \frac{H'(U,r)}{H(U,r)}+2(e^{Cr^{1-\sigma}}-1)
\geq \Phi_{\ell,\sigma}(U,0^+)=\kappa,
\end{align*}
which in turn yields
$$
\frac{H'(U,r)}{H(U,r)}\geq\frac{2}{r}\left[(\kappa+2)e^{-C r^{1-\sigma}}-2\right]\geq \frac{2\kappa}{r}\left(1-{C_1}r^{1-\sigma}\right),
$$
with ${C_1}=\left(1+\frac{2}{\kappa}\right)C$. Integrating we obtain
$$
\ln \frac{H(U,r_1)}{H(U,r)}\geq \ln \frac{r_1^{2\kappa}}{r^{2\kappa}}-C_2r_1^\sigma,
$$
and therefore
$$
H(U,r)\leq C_3  r^{2\kappa}\frac{H(U,r_1)}{r_1^{2\kappa}}.
$$
We now observe that either $r_1=1$, or $H(U,r_1)=r_1^{2\ell-2+2\sigma}$. In the former case we have $H(U,1)\leq C\left(\|U\|^2_{L^2(\Sa_1^+,|y|^a)}+C_\ell^2\right)$ by the $L^\infty$ bound on $U$, whereas in the latter we recall that, by assumption, $\kappa<\ell-1+\sigma$. Either ways, $$H(U,r_1)\leq C\left(\|U\|^2_{L^2(\Sa_1^+,|y|^a)}+C_\ell^2\right)r_1^{2\kappa},$$ which gives the desired conclusion.
\end{proof}

\begin{lemma}\label{l:W0} If $0\in\Gamma_\kappa^{(\ell)}(U)\ \text{for }\kappa<\ell-1+\sigma$, then
$$\mathscr W_\kappa(U,0^+)=0.$$
\end{lemma}
\begin{proof} By Lemma~\ref{L:2plus2sigma}, we know that
$$
\kappa=\Phi_{\ell,\sigma}(U,0^+)=\lim_{r\to 0^+}\tilde{N}(U,r)=2\lim_{r\to 0^+}\frac{D(U,r)}{H(U,r)}.
$$
Moreover, we infer from Lemma~\ref{l:optH} that $H(U,r)\leq Cr^{2\kappa}$. Hence,
\[
\lim_{r\to 0^+} \mathscr W_\kappa(U,r){=} \lim_{r\to 0^+}\frac{H(U,r)}{r^{2\kappa}} \left(\frac{D(U,r)}{H(U,r)} -  \frac{\kappa}2 \right)=0.\qedhere
\]
\end{proof}

\begin{dfn}\label{D:poly}
For $\kappa>0$ we denote by $\Pk$ the class of all parabolically $\kappa$-homogeneous polynomials $p_\kappa$ in $\Rnn\times(-\infty,0)$ such that
\begin{itemize}
\item[(i)] $\La p_\kappa = 0$;
\item[(ii)] $p_\kappa(x,0,t) \ge 0$;
\item[(iii)] $p_\kappa(x,-y,t) = p_\kappa(x,y,t)$;
\item[(iv)] $\kappa = 2 m$,\ $m\in \mathbb N$.
\end{itemize}
\end{dfn}

\begin{thrm}[Monneau type monotonicity formula]\label{T:monneau}
Let $U\in \Sp$ with $F$ satisfying \eqref{aF}, \eqref{aF-2}, \eqref{aF-3} for some $\ell\ge 4$ and a constant $C_\ell$.
Assume that $0\in \Sigma_\kappa(U)$ with $\kappa = 2m<\ell$, for $m\in \mathbb N$. For any $p_\kappa$ we define the \emph{Monneau type functional}
\begin{equation}\label{kM}
\Mk \overset{\rm def}{=} \frac{1}{r^{2\kappa + 2}} \int_{\Sa_r^+} (U - p_\kappa)^2\ \Gb y^a,\quad r\in (0,1).
\end{equation}
Then, for any $1-\ell-\kappa\leq \sigma<1$ there exists a constant $C''>0$, depending on $n, a, \ell, C_\ell, \sigma$, such that
\begin{equation}\label{kM'}
\frac{d}{dr} \Mk \ge - C'' \left(1 + \|U\|_{L^2(\Sa^+_1,\Gb y^a)} + \|p_\kappa\|_{L^2(\Sa^+_1,\Gb y^a)}\right) r^{-\sigma}.
\end{equation}
In particular, with $C = \frac{C''}{1-\sigma}$ the function $r\to \Mk + C r^{1-\sigma}$ is monotonically nondecreasing on $(0,1)$.
\end{thrm}

\begin{proof}
Letting $V = U- p_\kappa$. Notice that from (i) in Definition \ref{D:poly} we have in $\Sa_1^+$
\[
\mathscr L_a V = \mathscr L_a U - \mathscr L_a p_\kappa = F.
\]
From Remark \ref{R:freq} we know that
\[
\mathscr W_\kappa(p_\kappa,r) = \frac{H(p_\kappa,r)}{2r^{2\kappa}} \left(\tilde N(p_\kappa,r) - \kappa\right)  \equiv 0.
\]
We now use this information to show that
\begin{equation}\label{eq:SameW}
\mathscr W_\kappa(U,r) = \mathscr W_\kappa(V,r).
\end{equation}
In fact, we find from \eqref{kW}
\begin{align*}
\mathscr W_\kappa(U,r) & = \mathscr W_\kappa(U,r)  - \mathscr W_\kappa(p_\kappa,r) = \mathscr W_\kappa(V+p_\kappa,r)  - \mathscr W_\kappa(p_\kappa,r)
\\
& = \frac{1}{r^{2\kappa + 2}} \int_{\Sa_r^+} |t|\left(|\nabla V|^2 + 2\langle \nabla V,\nabla p_\kappa\rangle  \right)\ \Gb y^a
\\
&\qquad - \frac{\kappa}2 \frac{1}{r^{2\kappa + 2}} \int_{\Sa_r^+} \left(V^2 + 2 V p_\kappa\right)\ \Gb y^a
\\
& = \mathscr W_\kappa(V,r) + \frac{2}{r^{2\kappa + 2}} \int_{\Sa_r^+} |t| \langle \nabla V,\nabla p_\kappa\rangle   \Gb y^a -  \frac{\kappa}{r^{2\kappa + 2}} \int_{\Sa_r^+} V p_\kappa\ \Gb y^a
\\
& = \mathscr W_\kappa(V,r) + \frac{1}{r^{2\kappa + 2}} \int_{\Sa_r^+} V \left(Zp_\kappa - \kappa p_\kappa\right) \ \Gb y^a
\\
& = \mathscr W_\kappa(V,r),
\end{align*}
in view of the fact that $Zp_\kappa = \kappa p_\kappa$.
Since \eqref{kM} and \eqref{HU} give
\[
\mathscr M_\kappa(U,p_\kappa,r) = \frac{H(V,r)}{r^{2\kappa}},
\]
we obtain
\begin{align}\label{eq:M'}
\frac{d}{dr} \Mk & = \frac{H'(V,r)}{r^{2\kappa}} - \frac{2\kappa}{r^{2\kappa + 1}} H(V,r).
\end{align}
Using computations similar to the ones carried out in the proof of Lemmas \ref{L:alti} and \ref{L:IIU},  and applying Lemma~\ref{L:HU'}, we infer  that
\begin{align*}
H'(V,r)&=\frac{4}{r}I(V,r)\\
&=\frac{4}{r}\left\{D(V,r)-\frac{1}{r^2}\int_{\Sa_r^+}|t| V F\ \Gb y^{a} +\frac{1}{r^2}\int_{S_r} |t| V(x,0,t) \p^a_y V(x,0,t) \Gb (x,0,t)\right\}.
\end{align*}
Inserting this information in \eqref{eq:M'} yields
\begin{align*}
\frac{d}{dr} \Mk & = \frac{4}{r^{2\kappa +1}}\bigg\{D(V,r)-\frac{1}{r^2}\int_{\Sa_r^+}|t| V F\ \Gb y^{a} \\ &\qquad +\frac{1}{r^2}\int_{S_r} |t| V(x,0,t) \p^a_y V(x,0,t) \Gb (x,0,t)\bigg\}- \frac{2\kappa}{r^{2\kappa + 1}} H(V,r)\\
&= \frac{4}{r} \mathscr W_\kappa (V,r)- \frac{4}{r^{2\kappa +3}}\int_{\Sa_r^+}|t| V F\ \Gb y^{a}\\
&\qquad +\frac{4}{r^{2\kappa +3}}\int_{S_r} |t| p_\kappa(x,0,t) \p^a_y U(x,0,t) \Gb (x,0,t).
\end{align*}
We proceed to estimate each term in the last line. Using \eqref{eq:SameW}, and applying Theorem~\ref{T:weisspar} and Lemma~\ref{l:W0}, we infer that for a suitable choice of a constant $C$
$$
\mathscr W_\kappa (V,r)=\mathscr W_\kappa (U,r)\geq \mathscr W_\kappa(U,0^+)- C r^{2-2\sigma}=- C r^{2-2\sigma}.
$$
For the second term, we apply Cauchy-Schwarz's inequality, \eqref{aF3}, and Lemma~\ref{l:optH} to obtain
\begin{align*}
\frac{1}{r^{2\kappa +3}}\int_{\Sa_r^+}|t| V F\ \Gb y^{a}&\leq \frac{1}{r^{2\kappa +3}}\left(\int_{\Sa_r^+} V^2 \ \Gb y^{a}\right)^{1/2}\left(\int_{\Sa_r^+} t^2 F^2 \ \Gb y^{a}\right)^{1/2}\\
&\leq \frac{Cr}{r^{2\kappa +3}}\left(H(U,r)^{1/2}+H(p_\kappa,r)^{1/2}\right) r^{1+ \ell}\\
& \leq C\left(\|U\|_{L^2(\Sa_1^+,|y|^a)}+\|p_\kappa\|_{L^2(\Sa_1^+,|y|^a)}+1\right)r^{\ell-\kappa-1}\\
&\leq C\left(\|U\|_{L^2(\Sa_1^+,|y|^a)}+\|p_\kappa\|_{L^2(\Sa_1^+,|y|^a)}+1\right)r^{-\sigma}.
\end{align*}
Finally, to conclude, we observe that
$$
p_\kappa(x,0,t)\geq 0\quad\mbox{and } p_\kappa(x,0,t)\geq 0,
$$
so that
$$
\int_{S_r} |t| p_\kappa(x,0,t) \p^a_y U(x,0,t) \Gb (x,0,t)\geq 0.
$$
We thus conclude
\begin{align*}
\frac{d}{dr} \Mk & \geq - C r^{1-2\sigma} - C\left(\|U\|_{L^2(\Sa_1^+,|y|^a)}+\|p_\kappa\|_{L^2(\Sa_1^+,|y|^a)}+1\right)r^{-\sigma}\\
& \geq - C\left(\|U\|_{L^2(\Sa_1^+,|y|^a)}+\|p_\kappa\|_{L^2(\Sa_1^+,|y|^a)}+1\right)r^{-\sigma},
\end{align*}
as desired.
\end{proof}

\section{Structure of the singular set}\label{S:structure}
As before, we assume that the obstacle  $\psi \in H^{\ell, \ell/2}$ for some $\ell\geq 4$. Similarly to \cite{DGPT}, we  define the spatial dimension of the singular set based on the polynomial $p_{\kappa}^{(x_0, t_0)}$ as in  Proposition~\ref{prop:char-sing-point}.  For a singular point $(x_0, t_0) \in \Sigma_{\kappa}(U)$, we define
\begin{multline}\label{dim}
d_{\kappa}^{(x_0, t_0)}\overset{\rm def}{=}\dim\{ \xi \in \Rn \mid\, \langle \xi, \nabla_{x} \partial_{x}^{\alpha} \partial_t^{j} p_{\kappa}^{(x_0, t_0)}\rangle  =0\
\\
\text{for any $\alpha= (\alpha_1,\ldots, \alpha_n)$ and $j \geq 0$ such that $|\alpha| + 2j = \kappa-1$}\},
\end{multline}
which we  call as the spatial dimension of $\Sigma_{\kappa}(U)$ at $(x_0, t_0)$. Likewise, for  any $d=0, \ldots, n$, we define
\[
\Sigma_{\kappa}^{d}(U)= \{ (x_0, t_0) \in \Sigma_{\kappa}(U)\mid d_{\kappa}^{(x_0, t_0)}= d \}.
\]
In the case when $d=n$, i.e.,\  $(x_0, t_0) \in \Sigma_{\kappa}^n (U)$, the blow up limit $p_{\kappa}^{(x_0, t_0)}$    depends only on $y, t$ when $\kappa = 2m < \ell$. In such a case, $(x_0, t_0)$ is referred to as time-like singular point.  The proof of this fact is analogous to that of Lemma~12.10 in \cite{DGPT} (for the case $a=0$) and can be seen as follows.  Since in this case it holds
\[
\nabla_{x} \partial_{x}^{\alpha} \partial_{t}^{j} p_{\kappa}=0
\]
for all $|\alpha| + 2j = \kappa-1$, we have vanishing of $\partial_{x_i} p_{\kappa}$ on $\{y=0\}$.  Moreover, using  the fact that $\p_y^a \partial_{x_i} p_{\kappa}$ also vanishes identically on $\{y=0\}$ and $\partial_{x_i} p_{\kappa}$ is  $\La$ caloric, by the strong unique continuation property we obtain $\partial_{x_i} p_{\kappa} \equiv 0$ and hence $p_{\kappa}$ depends only on $y, t$.

\medskip

Now we recall the definition of space-like and time-like manifolds as in Definition 12.11 in \cite{DGPT}.

\begin{dfn}
We say that a $(d+1)$ dimensional manifold $\mathscr{S} \subset \Rn \times \R$ for $d=0, \ldots, n-1$  is space-like of class $C^{1, 0}$ if locally, after a rotation of coordinates, one can represent it as a graph
\[
(x_{d+1},\ldots, x_n)= g(x_1, \ldots, x_d, t),
\]
where $g, \nabla_x g$ are continuous.

\medskip

Likewise, a $n$-dimensional manifold $\mathscr{S} \subset  \Rn  \times \R$ is time-like of class $C^{1}$ if it can be locally represented as
\[
t= g(x_1,\ldots, x_n),
\]
where $g$ is $C^{1}$.

\end{dfn}

With the Monneau-type monotonicity formula as in Theorem~\ref{T:monneau} in hand,  we can repeat the arguments as in \cite{DGPT}  using the $L^{\infty}-L^{2}$ type estimates as in  Lemma~\ref{hr}  to assert  non-degeneracy of Almgren-Poon blowup at singular points and also uniqueness and continuous dependence of $\kappa$-homogeneous blowups at singular points. Then  by again arguing as  in \cite{DGPT},  using Whitney extension and  the implicit function theorem, we obtain the following structure theorem of the singular set based on spatial dimension of the singular  point  as defined in \eqref{dim}.

\begin{thrm}[Structure of the singular set]\label{structure theorem}
Let $U$ be a solution  to \eqref{epb2}, where $\psi \in H^{\ell, \ell/2}$ for some $\ell \geq 4$. Then for any $\kappa = 2m < \ell$, we have $\Gamma_{(k)}(U)= \Sigma_{\kappa}(U)$. Moreover, for every $d=0, \ldots, n-1$, the set $\Sigma_{\kappa}^{d}(U)$ is contained in a countable union of $(d+1)$-dimensional space-like $C^{1, 0}$ manifolds and $\Sigma_{\kappa}^{n}(U)$ is contained in a countable union of  time-like $n$-dimensional $C^{1}$ manifolds.

\end{thrm}

\section{Appendix}\label{S:appA}

In this appendix we collect the proofs of some of the auxiliary results in Section~\ref{S:poon}

\begin{proof}[Proof of Lemma~\ref{L:alti}]
To prove \eqref{alti} we observe that by the equation \eqref{gn} satisfied by $U$ in $\Sa_1^+$, we have that in $\Rnp$
\[
\La(U^2) = 2 U \La U - 2 |\nabla U|^2 y^a = 2 U F y^a - 2 |\nabla U|^2 y^a.
\]
This gives
\begin{equation}\label{energy}
\int_{\Rnp} |\nabla U|^2 \Gb y^a = \int_{\Rnp} U F \Gb y^a - \frac 12 \int_{\Rnp} \La(U^2) \Gb.
\end{equation}
The following computation can be justified rigorously considering the region
\[
\mathscr R_\ve = \{X \in \Rnp \mid y>\ve\},
\]
and then let $\ve\to 0^+$. One should keep in mind that the outer normal on $\p \mathscr R_\ve$ is $\nu = - e_{n+1}$. Integrating by parts we find
\begin{align*}
\int_{\Rnp} \La(U^2) \Gb & = 2 \int_{\Rnp} U U_t \Gb y^a - \int_{\Rnp} \operatorname{div}(y^a \nabla (U^2)) \Gb
\\
& = 2 \int_{\Rnp} U U_t \Gb y^a + 2 \int_{\Rn\times \{0\}} U \p^a_y U \Gb +  \int_{\Rnp} \langle \nabla(U^2),\nabla \Gb\rangle   y^a
\\
& = 2 \int_{\Rnp} U U_t \Gb y^a + 2 \int_{\Rnp} U \langle \nabla U,\frac{X}{2t}\rangle   \Gb y^a,
\end{align*}
where in the last equality we have used \eqref{rp} and the fact that
\[
\int_{\Rn\times \{0\}} U \p^a_y U \Gb = 0.
\]
The vanishing of this integral is proved as follows. We write
\[
\int_{\Rn\times \{0\}} U \p^a_y U \Gb = \int_{(\Rn\times \{0\}) \cap \{U>0\}} U \p^a_y U \Gb + \int_{(\Rn\times \{0\})\cap \{U = 0\}} U \p^a_y U \Gb.
\]
The first integral in the right-hand side vanishes since $\p^a_y U = 0$ on the set $(\Rn\times \{0\}) \cap \{U>0\}$. The integral on the set $(\Rn\times \{0\})\cap \{U = 0\}$ vanishes since $\p^a_y U\in C^{0,\frac{1-a}{2}}_{\rm loc}$ up to thin set $\{y=0\}$.

We conclude that
\[
\frac 12 \int_{\Rnp} \La(U^2) \Gb =  \int_{\Rnp} U \left(U_t + \langle \nabla U,\frac{X}{2t}\rangle  \right) \Gb y^a.
\]
From this formula, \eqref{z11} and \eqref{ismall} we conclude that
\[
\frac 12 \int_{\Rnp} \La(U^2) \Gb  = \frac{1}{2t} \int_{\Rnp}   U ZU \ \Gb y^{a} = \frac 1t i(U,t).
\]
Combining this equation with \eqref{energy}, we conclude that \eqref{alti} holds.
\end{proof}

In order to prove Lemmas \ref{L:HU'} and \ref{L:IU'} for every $\delta\in (0,1)$ we consider the following truncated quantities
\begin{equation}\label{HUd}
H_\delta(U,r) = \frac{1}{r^2} \int_{-r^2}^{-\delta r^2} h(U,t)  dt = \frac{1}{r^2} \int_{\Sa_r^+\setminus \Sa_{\delta r}^+} U^2\ \Gb y^{a} dX dt,
\end{equation}
and
\begin{equation}\label{IUd}
D_\delta(U,r) = \frac{1}{r^2} \int_{-r^2}^{-\delta r^2} d(U,t) dt = \frac{1}{r^2} \int_{\Sa_r^+\setminus \Sa_{\delta r}^+} |t| |\nabla U|^2  \ \Gb y^{a} dXdt.
\end{equation}
Consideration of these integrals is justified by the fact that for every $\delta\in (0,1)$ we have
\begin{equation}\label{Gbb}
\Gb\in L^\infty(\Rnp \times (-1,-\delta)).
\end{equation}

\begin{proof}[Proof of Lemma~\ref{L:HU'}]
Using our assumptions on $U$ we can proceed as in the proof of Lemma~6.5 in \cite{BG}. We thus skip most details and only refer to the relevant changes.
The first step is to recognize that for $t\in (-1,-\delta)$ one has
\begin{equation}\label{hU'}
h'(U,t) = \frac 1t \int_{\Rnp}   U ZU \ \Gb y^{a} = \frac 2t i(U,t).
\end{equation}
Again the proof of \eqref{hU'} can be rigorously justified by integrating on the region $\mathscr R_\ve$, where we know that \eqref{Gbb} holds,
and then let $\ve\to 0^+$ using \eqref{rp}, \eqref{gb}, \eqref{nc} and the assumptions on $U$ on the thin set $\{y=0\}$.

Substituting \eqref{energy} in \eqref{hU'} we have
\begin{equation}\label{hU'2}
t h'(U,t) = 2 d(U,t) + 2 t \int_{\Rnp} U F \Gb y^a.
\end{equation}
Using \eqref{HUd} we obtain from \eqref{hU'2}
\begin{equation}\label{H'delta}
H'_\delta(U,r) = 2r \int_{-1}^{-\delta} t h'(U,r^2 t) dt = \frac 4r D_\delta(U,r) - \frac{4}{r^3} \int_{\Sa_r^+\setminus \Sa_{\delta r}^+} U F |t| \Gb y^a.
\end{equation}
At this point we can argue as in the proof of Lemma~6.5 in \cite{BG} to pass to the limit as $\delta\to 0^+$ in \eqref{H'delta} and reach the desired conclusion for $H'(U,r)$.
\end{proof}

\begin{proof}[Proof of Lemma~\ref{L:IU'}]
For every $\delta\in (0,1)$ we have
\[
D_\delta(U,r) = \int_{-1}^{-\delta} d(U,r^2 t) dt.
\]
This gives
\begin{equation}\label{Ddeltaprime}
D'_\delta(U,r) = 2r \int_{-1}^{-\delta} t d'(U,r^2 t) dt = \frac{2}{r^3} \int_{-r^2}^{-\delta r^2} t d'(U,t) dt.
\end{equation}
We next compute $d'(U,t)$ for $-1<t<-\delta$. We are going to use the scalings \eqref{pardil} and \eqref{Ghom}.

Again, to make rigorous the following computation we should first consider integrals on the region $\mathscr R_\ve$, and then pass to the limit as $\ve\to 0^+$.
For $t\in (-1,-\delta)$ and $0<\la < 1/t$ we have from \eqref{dsmall}
\[
d(U,\la^2 t) = - \la^2 t \int_{\Rnp} |\nabla U(X',\la^2 t)|^2 \Gb(X',\la^2 t) (y')^a dX'.
\]
The change of variable $X' = \la X$ and \eqref{Ghom} give
\begin{align*}
d(U,\la^2 t) & = - \la^{n+a+1} \la^2 t \int_{\Rnp} |\nabla U(\la X,\la^2 t)|^2 \Gb(\la X,\la^2 t) y^a dX
\\
& = - \la^2 t \int_{\Rnp} \left(|\nabla U|^2 \circ \delta_\la\right)(X,t)\  \Gb(X,t) y^a dX.
\end{align*}
Recalling that
\[
\frac{d}{d\la} (f\circ \delta_\la)(X,t)\big|_{\la = 1} = Z f(X,t),
\]
if we differentiate with respect to $\la$ and set $\la = 1$ in the previous identity we find
\begin{align*}
2 t d'(U,t) & =  - 2 t \int_{\Rnp} |\nabla U|^2 \  \Gb y^a
 - t \int_{\Rnp} Z(|\nabla U|^2)\  \Gb y^a
 \\
 & = - t \int_{\Rnp} \big[Z(|\nabla U|^2) + 2 |\nabla U|^2\big] \  \Gb y^a
\end{align*}
Consider the vector fields $X_i = \frac{\p}{\p x_i}$, $i=1,\ldots,n$, $X_{n+1} = \frac{\p}{\p y}$. One easily verifies that the commutator $[X_i,Z] = X_i$, $i=1,\ldots,n+1$. This gives (using summation convention)
\[
Z(|\nabla U|^2) = 2 Z X_i u X_i U = 2 X_i Zu X_i U - 2 X_i U X_i U = 2 \langle \nabla(ZU),\nabla U\rangle   - 2 |\nabla U|^2.
\]
Substituting in the latter equation and integrating by parts and recalling that the outer unit normal on $\p \Rnp$ is $-e_{n+1}$, we find
\begin{align*}
2 t d'(U,t) & = - 2 t \int_{\Rnp} \langle \nabla(ZU),\nabla U\rangle     \Gb y^a
\\
& =  2 t \int_{\Rn \times\{0\}} \p^a_y U ZU\  \Gb  + 2 t \int_{\Rnp} ZU \operatorname{div}(y^a \Gb \nabla U)
\\
& =  2 t \int_{\Rn \times\{0\}} \p^a_y U ZU\  \Gb  + 2 t \int_{\Rnp} ZU \operatorname{div}(y^a\nabla U)\ \Gb\\
&\qquad +  2 t \int_{\Rnp} ZU \langle \nabla U,\nabla \Gb\rangle   y^a
\\
& =  2 t \int_{\Rn \times\{0\}} \p^a_y U ZU\  \Gb  - 2 t \int_{\Rnp} ZU F\ \Gb y^a
\\
&\qquad +  2 t \int_{\Rnp} ZU \langle \nabla U,\frac{X}{2t}\rangle   \Gb y^a + 2 t \int_{\Rnp} ZU U_t\ \Gb y^a
\\
& =  2 t \int_{\Rn \times\{0\}} \p^a_y U ZU\  \Gb  - 2 t \int_{\Rnp} ZU F\ \Gb y^a
\\
&\qquad +  \int_{\Rnp} (ZU)^2 \ \Gb y^a.
\end{align*}
We have thus proved the following formula for $t\in (-1,-\delta)$
\begin{equation}\label{dersmallen}
d'(U,t) = \frac{1}{2t} \int_{\Rnp} (ZU)^2 \ \Gb y^a - \int_{\Rnp} ZU F\ \Gb y^a + \int_{\Rn \times\{0\}} \p^a_y U ZU\  \Gb.
\end{equation}
Substituting now \eqref{dersmallen} in \eqref{Ddeltaprime} we
 obtain
 \begin{multline}\label{ex1}
 D'_\delta(U,r) =  \frac{1}{r^3} \int_{\Sa_r^+\setminus \Sa_{\delta r}^+} (ZU)^2 \Gb y^a - \frac{2}{r^3} \int_{\Sa_r^+\setminus \Sa_{\delta r}^+} t(ZU) F\Gb y^a \\+  \frac{2}{r^3}\int_{S_r \setminus S_{\delta r}} t \p^a_y U ZU\  \Gb.
\end{multline}
We claim that on the thin set $\{y=0\}$ we have
\begin{equation}\label{cl}
\p^a_y U ZU =0\quad\text{a.e.\  with respect}\ y^a dXdt.
\end{equation}
 We first note that $U$ restricted to $\{y=0\}$ is locally  Lipschitz continuous in $x, t$.  We also  have that for a.e $t$,  since  $\nabla U_t (\cdot, t)  \in L^{2}_{\rm loc} (\Rnp, y^a dX)$, therefore $U_t$ has a $L^{2}_{\rm loc}$ trace at $\{y=0\}$. Moreover  by a standard weak type argument using test  functions, we can show that such a trace is in fact bounded  because of the Lipschitz continuity of $U$ in $t$ and coincides with the weak time derivative of $U$ at $\{y=0\}$.
Now   on the set $\{U>0\}$, we have that $\lim_{y \to 0} y^a U_y=0$, hence a.e.\  we have
\[
\lim_{y \to 0} y^a U_y ZU =0\ \text{on $\{U>0\}$}.
\]
Then  on the set $\{U=0\} \cap \{y=0\}$, we note that
\[
ZU=0\ \text{a.e.}
\]
which again implies $\lim_{y \to 0} y^a U_y ZU =0$ a.e.  Therefore the claim \eqref{cl} follows.  Combined with \eqref{ex1}, it gives
\begin{equation}\label{ex2}
 D'_\delta(U,r) =  \frac{1}{r^3} \int_{\Sa_r^+\setminus \Sa_{\delta r}^+} (ZU)^2 \Gb y^a - \frac{2}{r^3} \int_{\Sa_r^+\setminus \Sa_{\delta r}^+} t(ZU) F\Gb y^a.
\end{equation}
At this point, we can argue as in the proof of Lemma~6.10 in \cite{BG} to reach the desired conclusion by letting $\delta \to 0^{+}$.
\end{proof}

\end{document}